\documentclass[12pt, reqno, twoside, letterpaper]{amsart}

\usepackage{paperstyle, graphicx}

\newcommand{\change}{}

\newcommand{\red}[1]{{\color{red}{#1}}}

\newcommand{\pfrac}[2]{\left(\frac{#1}{#2}\right)} 
\newcommand{\be}{\begin{equation}}
\newcommand{\ee}{\end{equation}}
\newcommand{\dalign}[1]{\[\begin{aligned} #1 \end{aligned}\]}

\newcommand{\E}{\mathbb{E}}
\newcommand{\PR}{\mathbb{P}}

\def\cH{\mathcal{H}}

\newcommand{\rr}{\mathbf{r}}
\renewcommand{\P}{\mathbb{P}}

\makeatletter
\renewcommand{\pmod}[1]{\allowbreak\mkern7mu({\operator@font mod}\,\,#1)}
\makeatother

\DeclareFontFamily{OT1}{rsfs}{}
\DeclareFontShape{OT1}{rsfs}{n}{it}{<-> rsfs10}{}
\DeclareMathAlphabet{\mathscr}{OT1}{rsfs}{n}{it}

\renewcommand{\cA}{\EuScript{A}}

\renewcommand{\cC}{\EuScript{C}}
\renewcommand{\cG}{\EuScript{G}}
\renewcommand{\cP}{\EuScript{P}}
\renewcommand{\cR}{\EuScript{R}}
\newcommand{\TrunSS}[1]{V_\cH(#1)}
\newcommand{\ccP}{\mathcal{P}}
\newcommand{\ccG}{\mathcal{G}}
\renewcommand{\dd}{\mathbf{d}}

\title[Large prime gaps and probabilistic models]
 {Large prime gaps and probabilistic models}
 
\author[W.~Banks]{William Banks}

\address{Department of Mathematics, 
 University of Missouri, 
 Columbia MO 65211, USA.}

\email{bankswd@missouri.edu}

\author[K.~Ford]{Kevin Ford}

\address{Department of Mathematics,
 University of Illinois, 1409 West Green St,
 Urbana, IL 61801, USA.}

\email{ford126@illinois.edu}

\author[T.~Tao]{Terence Tao}

\address{Department of Mathematics,
 UCLA, 405 Hilgard Ave,
 Los Angeles CA 90095, USA.}

\email{tao@math.ucla.edu}
 
\date{\today}

\begin{document}

\begin{abstract}
We introduce a new probabilistic model of the primes
consisting of integers that survive the sieving process
when a random residue class is selected for every prime
modulus below a specific bound. From a rigorous analysis of this model,
we obtain heuristic upper and lower bounds for the size of the
largest prime gap in the interval $[1,x]$. Our results are stated in terms of 
the extremal bounds in the interval sieve problem. The same 
methods also allow us to rigorously relate the validity of the
Hardy-Littlewood conjectures for an arbitrary set
(such as the actual primes) to lower bounds for the largest gaps
within that set.
\end{abstract}

\thanks{MSC Primary: 11N05; 11B83.}

\thanks{
\textbf{Keywords:} primes, prime gaps, Hardy-Littlewood conjectures, probabilistic models, sieves}

\thanks{\textbf{Acknowledgements:} 
The first author was supported by a grant from the University of
Missouri Research Board.
The second author was supported by National Science Foundation grants DMS-1501982 and DMS-1802139. The third author was supported by a Simons Investigator grant, the James and Carol Collins Chair, the Mathematical Analysis \&
Application Research Fund Endowment, and by NSF grant DMS-1764034. 
Part of the work was accomplished during a visit of the
the second author to the University of Missouri (August 2017),
UCLA (November 2018), the University of Cambridge (March 2019), 
the University of Oxford (March-June, 2019)
and the Institute of Mathematics
of the Bulgarian Academy of Sciences (June-July, 2019).
He thanks them all for their support.}

\maketitle



{\Large\section{Introduction}}

In this paper, we introduce a new probabilistic model $\cR \subset \N$
for the primes $\cP\defeq \{2,3,5,\ldots\}$
which can be analyzed rigorously
to make a variety of heuristic predictions.
In contrast to the well known prime model $\cC$ of Cram\'er~\cite{Cramer2} and
the subsequent refinement $\cG$ of Granville~\cite{Gr}, in which random sets
are formed by including positive integers with specific probabilities,
the model $\cR$ proposed here is comprised of integers that survive
the sieve when a random residue class is selected for every prime
modulus below a specific bound.
We determine the asymptotic behavior of the largest gap function,
$G_\cR(x)$, for the set $\cR$, where for any subset $\cA\subset \N$ we 
denote
\[ 
G_{\cA}(x)\defeq\max \{ b-a : [a,b] \subset [1,x]\text{~and~}
[a,b] \cap \cA=\varnothing \}.
\]
We conjecture that the primes $\cP$ have similar behavior.
Our bounds, given in Theorem~\ref{thm:conditional} below, are stated in terms of the 
extremal bounds in the interval sieve problem. 

At present, the strongest unconditional lower bound on $G_{\cP}(x)$
is due to Ford, Green, Konyagin, Maynard, and Tao~\cite{FGKMT}, who
have shown that\footnote{See Section \ref{notation-sec} for the
asymptotic notation used in this paper.}
\[
G_{\cP}(x)\gg\frac{\log x\,\log_2x\,\log_4x}{\log_3x},
\]
for sufficiently large $x$, with $\log_k x$ the $k$-fold iterated natural logarithm of $x$, whereas the strongest unconditional upper bound is
\[
G_{\cP}(x)\ll x^{0.525},
\]
a result due to Baker, Harman, and Pintz~\cite{BHP}. Assuming the Riemann Hypothesis, Cram\'er~\cite{Cramer1} showed that
\[
G_{\cP}(x) \ll x^{1/2} \log x.
\]

\bigskip


\subsection{Cram\'er's random model}
In 1936, Cram\'er~\cite{Cramer2} introduced a 
probabilistic model $\cC$ of primes, where each natural number $n\ge 3$
is selected for inclusion in~$\cC$
with probability $1/\log n$, the events $n \in \cC$ being jointly independent in $n$.
By Hoeffding's inequality (or Lemma~\ref{lem:Bennett-ineq} below), for any fixed $\eps>0$ one has
\be\label{eq:RH-Cramer}
\pi_\cC(x)\defeq|\{n\in\cC:n\le x\}|
=\int_2^x \frac{dt}{\log t}+O(x^{1/2+\eps})
\ee
with probability one.\footnote{See also \cite[eq. (5)]{Cramer2} 
for a more precise version of \eqref{eq:RH-Cramer}.}
The analogous statement for primes is equivalent to
the Riemann Hypothesis. 
In 1936, Cram\'er \cite{Cramer2} proved that
$\limsup_{x\to \infty} \frac{G_\cC(x)}{\log^2 x}=1$ almost surely, and remarked:``Obviously we may take this as a suggestion that, for the particular sequence of ordinary prime numbers $p_n$, some similar relation may hold.''
Later, Shanks \cite{Shanks} conjectured the stronger
bound $G_\cP(x) \sim \log^2 x$, also based on the analysis
of a random model very similar to Cram\'er's model. 
 This is a natural conjecture in light of the fact that
 \be\label{eq:Cramer-gap}
G_\cC(x) \sim \log^2 x
\ee
holds with probability one (although \eqref{eq:Cramer-gap}
doesn't appear to have been observed before).
In the literature, the statements $G_\cP(x)=O(\log^2 x)$
and $G_\cP(x) \asymp \log^2 x$ are sometimes referred to
as ``Cram\'er's conjecture.''
Several people have made refined conjectures,
e.g., Cadwell \cite{Cadwell} has suggested that $G_\cP(x)$
is well-approximated by $(\log x)(\log x-\log_2 x)$,
a conjecture which is strongly supported by numerical 
calculations of gaps.
We refer the reader to Granville \cite{Gr} or Soundararajan~\cite{SOUND} for additional information about
the Cramer model and subsequent developments.

Tables of prime gaps have been computed up to $10^{18}$ and beyond
(see \cite{Nicely-gaps}), thus
\[
\sup_{x\le 10^{18}} \frac{G_\cP(x)}{\log^2 x} \approx 0.9206,
\]
a consequence of the gap of size 1132 following the prime
 1693182318746371.
 See also Figure \ref{fig1} for a plot of $G(x)$ versus
 various approximations.

 \begin{figure}
 \includegraphics[height=3.65in]{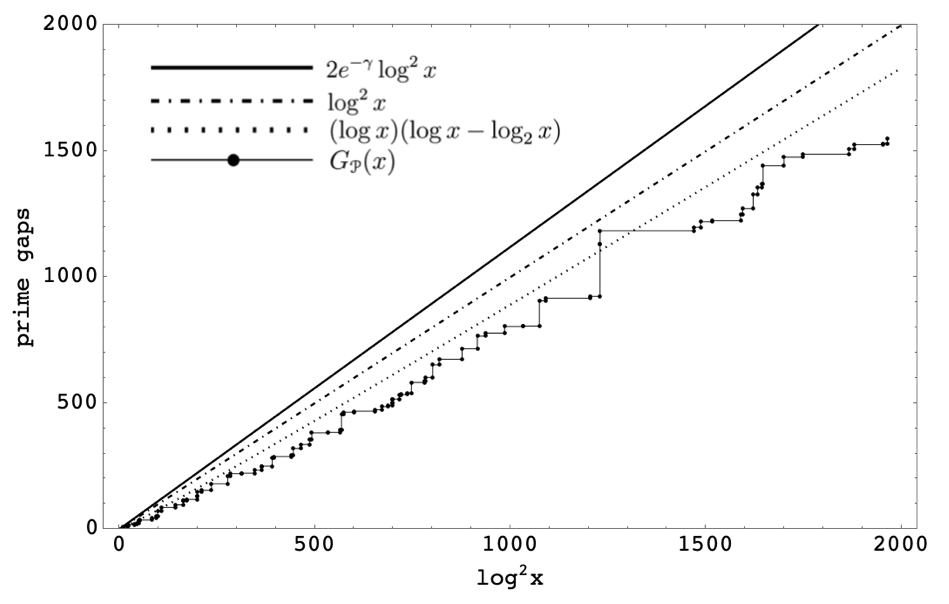}
\caption{$G_{\cP}(x)$ vs. various approximations}\label{fig1}
 \end{figure}

Despite its utility, the Cram\'er model has several well-documented weaknesses, the most dramatic one being that the model
does not predict the expected asymptotics for prime $k$-tuples.
Indeed, for \emph{any} finite set $\cH\subset\Z$, Cram\'er's model gives
\[
|\{n\le x: n+h\in\cC\text{~for all~}h\in\cH\}|
\sim \frac{x}{\log^{|\cH|} x} \qquad (x\to\infty)
\]
with probability one, whereas the analogous assertion for prime
numbers is false in general (for example, there is no integer $n$ such
that $n+h$ is prime for all $h\in\{0,1,2\}$).
The reason for the disparity is simple: for any prime $p$,
every prime other than $p$ must
lie in one of the residue classes $\{1,\ldots,p-1\}$ modulo~$p$
(we refer to this as the \emph{bias} of the primes modulo $p$),
whereas $\cC$ is equidistributed over \emph{all} residue
classes modulo $p$.

See~Pintz~\cite{Pintz} and Section \ref{sec:longer_intervals} below, for further discussion 
of flaws in the Cram\'er model.

%
%

\bigskip


\subsection{Granville's random model}
To correct this flaw in the Cram\'er model $\cC$,
Granville~\cite{Gr} altered the model, constructing a random
set $\cG$ as follows. For each interval $(x,2x]$
(with $x$ being a power of two, say), let $A$ be a parameter
such that $A=\log^{1-o(1)}x$ as $x\to\infty$, and
put $Q\defeq\prod_{p\le A} p$.
Discard those~$n$ for which $(n,Q)>1$, and select
for inclusion in $\cG$
each of the remaining integers $n\in (x,2x]$ 
with probability $\frac{Q/\phi(Q)}{\log n}$, where $\phi$ is the Euler totient function, the events $n \in \cG$ being jointly independent in $n$.
Since $\phi(Q)/Q$ is the density in $\ZZ$
of the set of integers coprime to $Q$,
this model captures the correct global distribution of primes;
that is, an analog of \eqref{eq:RH-Cramer} holds with $\cC$ replaced by $\cG$.
Unlike Cram\'er's model, however, Granville's model
also captures the bias of primes in residue classes modulo the primes $p\le A$. In particular, for any finite set $\cH$ of integers,
Granville's set satisfies the appropriate analog of the 
Hardy-Littlewood conjectures for counts of prime $k$-tuples
(see \eqref{eq:HL-asym} below).

In contrast with the Cram\'er model,
Granville's random set $\cG$ satisfies
\be\label{eq:Granville-gap}
G_\cG(x) \gtrsim \xi \log^2 x, \qquad \xi \defeq 2e^{-\gamma}=1.1229\cdots,
\ee
with probability one.
Granville establishes \eqref{eq:Granville-gap}
by choosing starting points $a$ with $Q\mid a$.
If $y\asymp \log^2 x$,
then there are about $y/\log y$ numbers $n\in [a,a+y]$ that are coprime
to every $p\le A$; this is a factor $\xi$ smaller than
the corresponding quantity for a random starting point $a$,
and it accounts for the difference between 
\eqref{eq:Cramer-gap} and \eqref{eq:Granville-gap}.
We elaborate on this idea in our analysis of
$G_\cR(x)$.

%
%

\bigskip


\subsection{A new probabilistic model for primes.}

Hardy and Littlewood \cite{HL} conjectured that the asymptotic relation
\be\label{eq:HL-asym}
|\{n\le x:n+h\in\cP\text{~for all~}h\in\cH\}|
\change{=\big(\fS(\cH)+o(1) \big)} \int_2^x \frac{dt}{\log^{|\cH|} t}
\ee
holds for any finite set $\cH\subset\Z$,
where $\fS(\cH)$ is the singular series given by
\be\label{eq:singular-series}
\fS(\cH)\defeq\prod_p\bigg(1-\frac{|\cH\bmod p|}p\bigg)
\bigg(1-\frac1p\bigg)^{-|\cH|}.
\ee
Note that the left side of \eqref{eq:HL-asym}
is bounded if $|\cH\bmod p|=p$ for
some prime $p$, since then for every integer $n$,
one has $p\,\mid\,n+h$ for some $h\in \cH$. 
\change{In this case, $\fS(\cH)=0$.}
 We say that $\cH$ is \emph{admissible}
if $|\cH\bmod p|<p$ for every prime $p$.

 To motivate our model set $\cR$, we first reinterpret
 \eqref{eq:HL-asym} probabilistically.
 The rapid convergence of the product
\eqref{eq:singular-series} implies that $\fS(\cH)$ is well
approximated by the truncation
\[
\fS_z(\cH)\defeq \prod_{p\le z}\bigg(1-\frac{|\cH\bmod p|}p\bigg)
\bigg(1-\frac1p\bigg)^{-|\cH|}=\TrunSS{z} \Theta_z^{-|\cH|},
\]
where
\be\label{V-Theta}
\TrunSS{z}\defeq\prod_{p\le z}\bigg(1-\frac{|\cH\bmod p|}p\bigg)\mand
\Theta_z\defeq\prod_{p\le z}\bigg(1-\frac1p\bigg).
\ee
We interpret $\TrunSS{z}$ as a product of local densities,
and $\Theta_z$ as a kind of global density.
In order to match the global density of primes as closely as
possible, we take $z=z(t)$ be the largest prime number for which
\change{$1/\Theta_{z(t)}\le\log t$};
this is well-defined for $t \geq e^2$, and by the prime number theorem we have
\be\label{eq:theta}
z(t)\sim t^{1/e^\gamma}\mand
\Theta_{z(t)}^{-1}=\log t+O(t^{-1/e^{\gamma}}).
\ee
It follows that
the right side of \eqref{eq:HL-asym} is 
\[
\sim \int_{e^2}^x\TrunSS{z(t)}\, dt.
\]
On the other hand, the quantity $V_\cH(z)$ can be written probabilistically as
\be\label{VHzSz}
\TrunSS{z}=\P(\cH\subset \cS_z),
\ee
where $\P$ denotes probability over a uniform choice of 
residue classes $a_p\bmod p$, for every prime $p$, with the random variables $a_p \bmod p$ being jointly independent in $p$, and $\cS_z$ is the random set
\be\label{eq:Sz}
\cS_z\defeq\Z \setminus \bigcup_{p\le z} (a_p\bmod p).
\ee
Thus, $\cH\subset \cS_z$ is the event
that $\cH$ survives sieving by random residue
classes modulo primes $p\le z$.
Consequently, \change{for admissible $\cH$,} \eqref{eq:HL-asym} takes the form
\[
|\{n\le x:n+h\in\cP\text{~for every~}h\in\cH\}|
\sim \int_{e^2}^x\P(\cH\subset \cS_{z(t)})\,dt.
\]
Thus, \eqref{eq:HL-asym} asserts that
the probability that a random shift of $\cH$ lies in $\cP$
is asymptotically the same as the probability that $\cH$
lies in a randomly sifted set.

Motivated by this probabilistic interpretation of \eqref{eq:HL-asym}, we now define
\be\label{eq:B-model-defn}
\cR\defeq\{n \geq e^2:n\in\cS_{z(n)}\}
\ee
as our random set of integers. Note that the number of primes being
sieved out increases as $n$ increases in order to mimic the slowly
decreasing density of the primes. This can be compared with the
description of $\cP$ using the sieve of Eratosthenes,
in which $z(n)$ is replaced by $n^{1/2}$ and the $a_p$ are replaced by $0$.

We believe that the random set $\cR$ is a useful model for primes,
especially for studying local statistics such as gaps.
On the other hand, the analysis of $\cR$ presents more difficulties
than the analysis of $\cC$ or $\cG$, owing to the more complicated
coupling between events such as $n_1\in \cR$ and $n_2\in\cR$ for
$n_1 \neq n_2$.


\subsection{Large gaps from the model}
 The behavior of 
$G_\cR(x)$ is intimately tied to extremal properties
of the \emph{interval sieve}. To describe this connection, for any $y \geq 2$ let $W_y$ denote the (deterministic) quantity
\be\label{eq:By-defn}
W_y \defeq \min\big|[0,y]\cap\cS_{(y/\log y)^{1/2}}\big|,
\ee
where $\cS_z$ is defined \change{in} \eqref{eq:Sz} and the minimum
in \eqref{eq:By-defn} is taken over all choices of the residue classes
$\{a_p\bmod p: p\le(y/\log y)^{1/2}\}$. At present, the sharpest known bounds on $W_y$ are
\be\label{eq:By-bounds}
\change{(4+o(1)) \frac{y\log_2y}{\log^2 y}
\le } W_{y} \le \frac{y}{\log y}+O\pfrac{y\log_2y}{\log^2 y},
\ee
the lower bound being a consequence of Iwaniec's theory
 (see \cite[Theorem 12.14]{FI} or~\cite{I71}) of the linear sieve,
and the upper bound resulting from the particular choice
$a_p\defeq 0\bmod p$ for all primes $p\le (y/\log y)^{1/2}$.
There is a folklore conjecture that the upper bound in \eqref{eq:By-bounds} is
closer to the truth. The problem of bounding $W_y$
belongs to a circle of problems centered on the question
about the maximum number of primes in 
some interval of length $x$; see e.g., \cite{HensRich} and\cite{ER}.

\begin{theorem}[Asymptotic for largest gap in the random model]\label{thm:conditional}
Put
\be\label{eq:gdef}
g(u) \defeq \max \{ y : W_y \log y \le u \}
\ee
\change{and define $\xi\defeq 2e^{-\gamma}=1.1229\ldots$. For any $\eps>0$, 
with probability one, we have
\[
g((\xi-\eps)\log^2 x) \leq G_\cR(x) \le g((\xi+\eps)\log^2 x)
\]
for all large $x$.}
\end{theorem}

The function $g(u)$ is evidently increasing, and by \eqref{eq:By-bounds} we see that 
\be\label{g-bounds} 
\change{(1+o(1)) u \le g(u) \le (1+o(1)) \frac{u\log u}{4\log_2 u} \qquad
(u\to\infty)}
\ee
and so Theorem~\ref{thm:conditional} implies that \change{for every $\eps>0$,}
almost surely \change{we have}
\be\label{GBx_extrema}
\change{(\xi-\eps)} \log^2 x \le G_\cR(x) \le \change{(\xi+\eps)} \frac{\log^2 x \log_2 x}{2\log_3 x}
\ee
\change{for all large $x$.}

\change{It seems likely that $g((\xi \pm \eps)\log^2 x)\to g(\xi\log^2 x)$ as $\eps\to 0$, although we cannot prove this. Assuming this,}
Theorem~\ref{thm:conditional}
 leads us to the following prediction for gaps between primes:

\begin{conjecture}[Asymptotic for largest gap in the primes]\label{conj:Gxhx}
We have 
\[
 \change{G_\cP(x) \sim g(\xi \log^2 x) \qquad (x\to\infty).}
\]
\end{conjecture}

 Assuming the previously mentioned folklore conjecture that the lower bound in \eqref{g-bounds} is asymptotically tight in the sense that $g(u)\sim u$ as $u\to\infty$, we are then led to the prediction that
\[
G_\cP(x) \sim \xi \log^2 x \change{\qquad (x\to\infty).}
\]
This matches the lower bound \eqref{eq:Granville-gap}
for the gap in the Granville model $\cG$.


\subsection{Hardy-Littlewood from the model}

It has been conjectured that a much more precise version
of \eqref{eq:HL-asym} holds (see, e.g., Montgomery and
Soundararajan~\cite{MS}), namely:
\be\label{eq:HL-precise}
|\{n\le x:n+h\in\cP\text{~for all~}h\in\cH\}|
=\fS(\cH) \int_2^x \frac{dt}{\log^{|\cH|}t}+O(x^{1/2+\eps}).
\ee
There is some computational evidence for this strong estimate for certain small sets $\cH$; see
Section \ref{sec:comments_HL}.
Granville's model set $\cG$, by contrast, satisfies the
analogous relation with an error term 
that cannot be made smaller than $O(x/\log^{|\cH|+1} x)$.
This occurs because $\cG$ is only capturing the bias of $\cP$ modulo
 primes $p\le A$; that is, the set $\cG$ satisfies the analog
 of \eqref{eq:HL-precise} with $\fS(\cH)$ replaced by 
 $\fS_A(\cH)$.

The model set $\cR$ given by \eqref{eq:B-model-defn}
has been designed with the Hardy-Littlewood conjectures in mind. We establish
a uniform analog of \eqref{eq:HL-precise} that holds in a
wide range of $\cH$. 

\begin{theorem}[Hardy-Littlewood conjecture for the random model]\label{thm:UHL}
Fix $c\in[1/2,1)$ and $\eps>0$.
 Almost surely, we have
\[
|\{n\le x:n+h\in\cR\text{~for all~}h\in\cH\}|
=\fS(\cH)\int_{2}^{x} \frac{dt}{\log^{|\cH|} t}+
O\big(x^{1-\frac{1-c}{8c^2-2c}+\eps}\big)
\]
uniformly for all admissible tuples $\cH$ satisfying
$|\cH| \le \log^c x$ and in the range 
$\cH \subset [0,\exp( \frac{\log^{1-c} x}{\log_2 x} )]$.
\end{theorem}

In particular, when \change{$c=\frac12$} the error term
is $O(x^{1/2+o(1)})$, which matches \eqref{eq:HL-precise} provided that
$\cH \subseteq [0,\exp\{\frac{\log^{1/2} x}{\log_2 x} \}]$
and $|\cH|\le \log^{1/2} x$.
\change{As we will invoke the Borel-Cantelli lemma in the proof,
the constant implied by the $O-$symbol exists almost surely, but we
cannot give any uniform bound on it. This remark applies to the next result as well.}

For the special case $\cH=\{0\}$ we have the following
more precise statement.

\begin{theorem}[Riemann hypothesis for the random model]\label{thm:RHB}
Fix $c>3/2$. Almost surely, we have
\[
|\{n\in\cR:n\le x\}|=\int_2^x\frac{dt}{\log t}+O(x^{1/2} \log^c x).
\]
\end{theorem}

Similar results can be obtained for any \emph{fixed} tuple $\cH$; we leave this to the interested reader.


\subsection{Large gaps from Hardy-Littlewood}

The results stated above have a partial deterministic converse.
We show that \emph{any} set of integers
that satisfies a uniform analogue of the Hardy-Littlewood conjecture
\eqref{eq:HL-precise} has large gaps. The maximal length of the gaps
depends on the range of uniformity of 
\eqref{eq:HL-precise}, and comes close to order $\log^2 x$
with a strong uniformity assumption. Our result
extends a theorem of Gallagher \cite{GAL}, who showed
that \change{if, for every fixed $k\in \NN$ and real $c>1$, the primes obey the Hardy-Littlewood conjectures uniformly for every admissible $k$-tuple $\cH \subset [0,c\log x]$}, then the gaps normalized by $\frac{1}{\log x}$ enjoy an exponential distribution asymptotically.
His approach applies to any set $\cA$ in place of the primes $\cP$.

\begin{theorem}[Hardy-Littlewood implies large gaps]\label{thm:HLgaps}
Assume $\frac{2\log_2 x}{\log x} \le \kappa \le 1/2$ and that
 $\cA \subset \N$ satisfies the
Hardy-Littlewood type conjecture
\be\label{eq:HLA}
|\{n\le x:n+h\in\cA\text{~for all~}h\in\cH\}|
=\fS(\cH) \int_{2}^{x} \frac{dt}{\log^{|\cH|} t}+
O(x^{1-\kappa})
\ee
uniformly over all tuples $\cH \subset [0,\log^2 x]$
with $|\cH| \le \frac{\kappa \log x}{2\log_2 x}$.
Then 
\[
G_\cA(x) \gg \frac{\kappa \log^2 x}{ \log_2x}
\]
for all large $x$, where the implied constant is absolute.
\end{theorem}

We also have the following variant of Theorem~\ref{thm:HLgaps}, which has a stronger conclusion but requires a uniform Hardy-Littlewood conjecture for larger tuples (of cardinality as large as $\log x \log_2 x$); on the other hand, this conjecture is only needed in a certain averaged sense.

\begin{theorem}[Averaged Hardy-Littlewood implies large gaps]\label{thm:HLgaps-avg}
Fix $0<c<1$.
Suppose that $\cA \subset \N$ satisfies the averaged
Hardy-Littlewood type conjecture
\be\label{eq:HLA*}
\ssum{\cH \subset [0,y] \\ |\cH|=k}
|\{n\le x:n+h\in\cA\text{~for all~}h\in\cH\}|
=\ssum{\cH \subset [0,y] \\ |\cH|=k} 
\int_{2}^{x} \frac{\fS_{z(t)}(\cH)}{\log^k t}\,dt+
O(x^{1-c})
\ee
uniformly for $k \leq \frac{Cy}{\log x}$ and $\log x \leq y \leq (\log^2 x)\log_2 x$, where $C$ is a sufficiently large absolute constant.
Then
\[
G_{\cA}(x) \ge g((c\,\xi - o(1)) \log^2 x)
\qquad(x\to\infty),
\]
where $g$ is defined in \eqref{eq:gdef}.
\end{theorem}
 
One could combine Theorem~\ref{thm:UHL} with Theorem~\ref{thm:HLgaps} (taking $\kappa\defeq (\log x)^{c-1+\eps}$ with
fixed $c<1$, say) to obtain results similar to Theorem~\ref{thm:conditional}. However, the conclusion is considerably weaker than that of Theorem~\ref{thm:conditional}, and it does not appear that this approach is going to come close to recovering 
the bounds we obtain using a direct argument.

Below we summarize, in rough form,
the various results and conjectures for the primes $\cP$, the various random models $\cC, \cG, \cR$ for the primes, and for arbitrary sets $\cA$ obeying a Hardy-Littlewood type conjecture:

\bigskip

\begin{tabular}{cll}
\hline
Set & Hardy-Littlewood conjecture? & Asymptotic largest gap up to $x$ \\
\hline
$\cC$ & No (singular series is missing) &
$\mathop{\sim\log^2 x}\limits^{\phantom{.}}$ \\
$\cG$ & Yes (with weak error term) & $g( (\xi \pm o(1)) \log^2 x)$ \\
$\cR$ & Yes (with error $O(x^{1-c})$) & $g( (\xi \pm o(1)) \log^2 x)$ \\
$\cP$ & Yes (conjecturally) & $\sim\xi \log^2 x$ (conjecturally) \\
$\cA$ & Assumed (error $O(x^{1-c})$) & $\gg c \frac{\log^2 x}{\log_2 x}$ \\
$\cA$ & Assumed on average (error $O(x^{1-c})$) &
$\mathop{\gtrsim}\limits_{\phantom{.}} g((c\,\xi - o(1)) \log^2 x)$\\
& \change{for tuples of size up to $(\log x)\log_2 x$} &
\\
\hline
\end{tabular}

\bigskip
Of course, one can combine this table's conclusions with the unconditional bounds in \eqref{g-bounds}, or the conjecture $g(u) \sim u$, to obtain further rigorous or predicted upper and lower bounds for the largest gap.

\subsection{Open Problems}

\begin{enumerate}
\item Improve upon the bounds \eqref{eq:By-bounds}; 
alternatively, give some heuristic reason for why the
upper bound in \eqref{eq:By-bounds} should be closer to the truth.
\item Show that $g(a)\sim g(b)$ whenever $a\sim b$. 
This will clean up the \change{statement of Theorem~\ref{thm:conditional}.}
\item Analyze the distribution of large gaps between 
special elements of $\cR$. For example, what is the largest gap between elements of $\{n: n\in \cR,n+2\in \cR \}$ below $x$?
This should be a good predictor for the maximal gap between pairs of 
twin primes and likely will involve a different extremal sieve problem. 
\end{enumerate}


\subsection{Plan of the paper}
Following further remarks and background inequalities in
Sections \ref{sec:background-remarks} and \ref{sec:prelim}, we prove Theorems \ref{thm:UHL} and \ref{thm:RHB}
in Section \ref{sec:UHL_model} using first and second moment
bounds. 
 Section \ref{sec:gaps-randomly-sieved} and \ref{sec:random-large-primes} contain probability
estimates on $|[0,y] \cap \cS_w|$ for various ranges of $w$.
These are then used to prove Theorem~\ref{thm:conditional} in
Section \ref{sec:large-gaps-from-model} and Theorems
\ref{thm:HLgaps} and \ref{thm:HLgaps-avg} in Section 
\ref{sec:HLA}.
In Section \ref{sec:interval-sieve}, we connect
the interval sieve problem to the problem of ``exceptional zeros,'' made explicit in Theorem~\ref{thm:exceptional_zeros}; this is proved in Section \ref{sec:exceptional_zeros}.


{\Large\section{Background and Further Remarks}
\label{sec:background-remarks}}

The discussion here is not needed for the proofs of the main
theorems and may be omitted on the first reading.

\medskip


\subsection{Remarks on the Hardy-Littlewood conjectures}\label{sec:comments_HL}
%

\change{For any $\cH \subseteq [0,y]$, we have
$\fS(\cH) \le e^{O(|\cH|\log_2 y)}$}
(see Lemma~\ref{lem:SS} below),
and thus when $y \le (\log x)^{O(1)}$, the main terms in 
\eqref{eq:HL-precise} and \eqref{eq:HLA} are smaller than one for $c_1 \frac{\log x}{\log_2 x} \le |\cH|\le \exp\{ (\log x)^{c_2} \}$, 
where $c_1,c_2>0$ are appropriate constants.
Therefore, we cannot have a genuine asymptotic 
when $|\cH| > c_1 \frac{\log x}{\log_2 x}$.

In the case of primes, it may be the case that
\eqref{eq:HL-precise} fails
when $|\cH|> \frac{\log x}{\log_2 x}$ owing to
 potentially large fluctuations in both the size of $\fS(\cH)$
and in the prime counts themselves.
We note that 
Elsholtz \cite{Els} has shown that for any $c>0$,
the left side of \eqref{eq:HL-precise} is bounded by
$$O\left(x \exp\left( - (\tfrac14+o(1)) \frac{\log x \log_3 x}{\log_2 x} \right)\right)$$ 
when $|\cH|\ge c\log x$, where the implied function $o(1)$ depends on $c$. 
On the other hand, there are admissible tuples with $|\cH|\ll \log x$ for which the left side of \eqref{eq:HL-precise} is zero
(see \cite{Els} for a construction of such $\cH$).

Our assumption in Theorem~\ref{thm:HLgaps-avg} is more speculative,
in light of the above remarks, since we need
to deal with tuples $\cH$ satisfying $k=|\cH| > \log x$.
Also, simply considering subsets $\cH$ of the primes in $(y/2,y]$
(which are automatically admissible),
we see that
 there are at least $(\frac{y}{k\log y})^k>(\log x)^{k/2}$
tuples $\cH$ in the summation, and this means that 
when $k>\log x$,
\eqref{eq:HLA*} implies a great deal of cancellation in
the error terms of \eqref{eq:HLA} over tuples $\cH$.

In a few special cases, e.g., $\cH=\{0,2\}$, $\cH=\{0,2,6\}$,
and $\cH=\{0,4,6\}$, there is extensive numerical evidence
(cf.\ \cite[pp.~43--44, 62--64]{HL}, \cite{Nicely-95}, \cite{LeikWald}, 
\cite{Nicely-04}, \cite{Nicely-web}) in support of the
conjecture \eqref{eq:HL-precise} with such a strong error term\footnote{Most of this work appears only on web pages, rather than in
books or journals.}.
 Note that the special case of \eqref{eq:HL-precise}
with $\cH=\{0\}$ is equivalent to the Riemann Hypothesis.
Theorem~\ref{thm:UHL} makes plausible the notion that 
\eqref{eq:HL-precise} may hold uniformly 
for $\cH \subset [0,Y]$ \change{with} $|\cH|\le K$, where $Y,K$ 
are appropriate functions of $x$.

\medskip


\subsection{The cutoff $z(t)$}
In \cite{Polya}, P\'olya suggests using a truncation
$x^{1/e^\gamma}$ to justify the Hardy-Littlewood conjectures.
The observation that the cutoff $\sqrt{x}$ leads to
erroneous prime counts was made by Hardy and Littlewood
\cite[Section 4.3]{HL} and is occasionally referred to as
``the Mertens Paradox'' (see \cite{MontWagon}).
In discussing the probabilistic heuristic for counting the number of primes below $x$, Hardy and Littlewood write (here $\varpi$ denotes a prime) ``One might well replace $\varpi<\sqrt{n}$ by $\varpi < n$, in which case we should obtain a probability half as large. This remark is in itself enough to show the unsatisfactory character of the argument'' and later
``\emph{Probability} is not a notion of pure mathematics, but of philosophy or physics.''


\medskip\subsection{Connection to Jacobsthal's function}
Any improvement of the lower bound in \eqref{eq:By-bounds} 
leads to a corresponding improvement of the known upper bound on Jacobsthal's
function \change{$J(w)$}, which we define to be the largest gap which occurs
in the set of integers that have no prime factor \change{$\le w$.}
Equivalently, \change{$J(w)$} is the largest gap in \change{$\cS_w$.}
Iwaniec~\cite{I71} proved that \change{$J(w)\ll w^2$} using
his linear sieve bounds.
Using Montgomery and Vaughan's explicit version of the Brun-Titchmarsh
inequality \cite{MV-BT}, the cardinality of the set $\cS_w(y)\defeq[0,y]\cap\cS_w$ for $w > (y/\log y)^{1/2}$ can be bounded from below by
\begin{align*}
|\cS_w(y)|&\ge |\cS_{(y/\log y)^{1/2}}(y)| - \sum_{(y/\log y)^{1/2}<p\le w} |\cS_{(y/\log y)^{1/2}}(y) \cap (a_p \bmod p)| \\
&\geq W_y-\sum_{(y/\log y)^{1/2}<p\le w}\frac{2y/p}{\log(2y/p)}. 
\end{align*}
If the right side is positive, it follows that \change{$J(w)<y$.}
Suppose, for example, that $W_y \ge \alpha y/\log y$ for large $y$, where $0<\alpha\le 1$ is fixed. 
\change{ Mertens' estimates then imply that}
\[
\change{J(w) \ll w^{1+e^{-\alpha/2}+o(1)} \qquad (w\to\infty)},
\]
\change{which improves Iwaniec's upper bound.}

We remark that all of the unconditional lower bounds on
$G_\cP(x)$, including the current record \cite{FGKMT}, have utilized the simple inequality
$G_\cP(x)\ge J(y)$, where $y\sim \log x$.

\medskip


\subsection{The interval sieve problem and exceptional zeros}
\label{sec:interval-sieve}

The problem of determining $W_y$
asymptotically is connected with the famous problem
about exceptional zeros of Dirichlet $L$-functions
(also known as Siegel zeros or Landau-Siegel zeros);
see, e.g., \cite[Sections 14, 20, 21, 22]{Dav}
for background on these
and \cite{Iw02} for further discussion.

\begin{definition} 
We say that \emph{exceptional zeros exist} if there 
is an infinite set $\cE\subset \NN$, 
such that for every $q\in \cE$ there is a 
real Dirichlet character $\chi_q$ and a zero $1-\delta_q$
with $L(1-\delta_q,\chi_q)=0$ and $\delta_q=o(1/\log q)$
as $q\to\infty$.
\end{definition}

\begin{theorem}\label{thm:exceptional_zeros}
Suppose that exceptional zeros exist. Then
\[
\liminf_{y\to\infty} \frac{W_y}{y/\log y}=0 \mand
\limsup_{u\to \infty} \frac{g(u)}{u}=\infty.
\]
Hence, we almost surely have
\[
\limsup_{x\to \infty} \frac{G_\cR(x)}{\log^2 x}=\infty
\]
and Conjecture \ref{conj:Gxhx} implies that
\[
\limsup_{x\to \infty} \frac{G_\cP(x)}{\log^2 x}=\infty.
\]
\end{theorem}

Our proof of Theorem~\ref{thm:exceptional_zeros},
given in Section \ref{sec:exceptional_zeros},
is quantitative, exhibiting an upper bound for $W_y$
in terms of the decay of $\delta_q$. Siegel's theorem
\cite[Sec. 21]{Dav} implies that $\frac{\log 1/\delta_q}{\log q} \to 0$,
but we cannot say anything about the rate at which this occurs 
(i.e., the bound is \emph{ineffective}).
If the rate of decay to zero is extremely slow, then our proof
shows that, infinitely often, $W_y=f(y) \frac{y\log_2 y}{\log y}$, with $f(y)\to \infty$ extremely slowly. Consequently,
$G_\cR(x)$ is infinitely often close to the upper bound
in \eqref{GBx_extrema}.

The related quantity
\[
\widetilde{W}_y\defeq\max | S_{\sqrt{y}} \cap [0,y] |,
\]
is known by the theory of upper bound sieves to satisfy 
$\widetilde{W}_y \le \frac{2y}{\log y}$
(see, e.g., \cite{MV}),
and it is well known that an improvement 
of the constant two
 would imply that exceptional zeros
do not exist; see, e.g., Selberg's paper \cite{Selberg72}. Theorem~\ref{thm:exceptional_zeros} (in the contrapositive) similarly asserts that an improvement of the constant zero in the trivial lower bound
$W_y \geq 0\cdot\frac{y}{\log y}$ implies that exceptional zeroes do not exist. 
\change{Extending our ideas and those of Selberg, Granville \cite{Granville-intsieve}
has recently shown that if exceptional zeros exist, then for any real $r>1$,
\begin{align*}
\liminf_{y\to \infty} \frac{\min_{(a_p)} |[0,y] \cap \cS_{y^{1/r}}|}{e^{-\gamma} y/\log y^{1/r}} &=f(r), \\
\limsup_{y\to \infty} \frac{\max_{(a_p)} |[0,y] \cap \cS_{y^{1/r}}|}{e^{-\gamma} y/\log y^{1/r}} &=F(r),
\end{align*}
where $f,F$ are the lower and upper linear sieve functions.
In particular, $f(r)=0$ for $r\le 2$ and $f(r)>0$ for $r>2$.
}

It is widely believed that exceptional zeros do not exist,
and this is a famous unsolved problem.
Theorem~\ref{thm:exceptional_zeros} indicates that to
fully understand $W_y$, it is necessary to
solve this problem. Iwaniec's lectures
\cite{Iw02} give a nice overview of
the problem of exceptional zeros, attempts to prove that
they do not exist, and 
various consequences of their existence.
In \change{the paper} \cite{F19}, the second author shows that
if there is a sequence of moduli $q$ with $\delta_q\ll (\log q)^{-2}$,
then one can deduce larger lower bounds for 
\change{$J(w)$} and $G_\cP(x)$ than are currently known unconditionally.

\medskip


\subsection{Primes in longer intervals}
\label{sec:longer_intervals}

With probability one, the Cram\'er model $\cC$ also satisfies
\be\label{eq:Selberg}
\pi_{\cC}(x+y)-\pi_{\cC}(x)\sim\frac{y}{\log x}
\ee
as long as $x\to\infty$, \change{$y\le x$}, and $y/\log^2 x\to\infty$.
However, Maier~\cite{Maier} has shown that 
the analogous statement for primes is false, namely
that for any fixed $A>1$
one has
\be\label{Maier}
\liminf\limits_{x\to\infty}
\frac{\pi(x+(\log x)^A)-\pi(x)}{(\log x)^{A-1}}<1
\quad\text{and}\quad\limsup\limits_{x\to\infty}
\frac{\pi(x+(\log x)^A)-\pi(x)}{(\log x)^{A-1}}>1.
\ee
The disparity between \eqref{eq:Selberg} and \eqref{Maier}
again stems from 
 the uniform distribution of $\cC$ 
in residue classes modulo primes.
 Both models $\cG$ and $\cR$ satisfy the analogs of \eqref{Maier}; we omit the proofs.
Moreover,
the ideas behind Theorem~\ref{thm:conditional} can be used to
sharpen \eqref{Maier}, by replacing the right sides of
the inequalities by quantities defined in terms of 
 the extremal behavior of 
$|[0,y] \cap S_{y^{1/u}}|$ for fixed $u>1$; we refer the reader to 
\cite[Exercise 30.1]{Kou-book} for details.
The authors thank Dimitris Koukoulopoulos for this observation.

By contrast, on the Riemann Hypothesis,
Selberg~\cite{Selberg} showed that 
\[\pi(x+y)-\pi(x)\sim\frac{y}{\log x}\]
holds for \emph{almost all} $x$ provided that \change{$y=y(x)\le x$} satisfies $y/\log^2 x \to \infty$ as $x \to \infty$.

\change{On a related note, Granville and Lumley \cite{GL} have 
developed heuristics and conjectures
concerning the \emph{maximum} number of primes $\le x$ lying in intervals of length 
$L$, where $L$ varies between $\log x$ and $\log^2 x$}.


\subsection{Remarks on the singular series and prime gaps}

If $y$ is small compared to 
$x$, the difference $\pi_\cC(x+y)-\pi_\cC(x)$ is
a random variable with (essentially) a binomial distribution.
Letting $y\to\infty$ with $y/\log x$ fixed, the result is a
Poisson distribution: for any real $\lambda>0$ and
any integer $k\ge 0$, we have
\be\label{eq:Gallagher-says}
\big|\big\{m\le x:\pi_\cC(m+\lambda\log m)-\pi_\cC(m)=k\big\}\big|
\sim e^{-\lambda}\frac{\lambda^k}{k!}x\qquad(x\to\infty)
\ee
with probability one. In particular, using $\cC$ as
a model for the primes $\cP$, this leads to the conjecture
that
\be\label{eq:cramerexp}
\lim_{x\to\infty}
\pi(x)^{-1}\big|\{p_n\le x:p_{n+1}-p_n\ge \lambda\log p_n\}\big|
=e^{-\lambda}\qquad(\lambda>0).
\ee
 Gallagher~\cite{GAL} 
showed that if the Hardy-Littlewood conjectures
\eqref{eq:HL-asym} are true uniformly for \change{$\cH\subset [0,\log^2 x]$}
with fixed cardinality $|\cH|$,
then \eqref{eq:cramerexp} follows. His analysis 
relies on the relation
\be\label{eq:Gallagher}
\sum_{\substack{\cH\subset[0,y]\\|\cH|=k}}\fS(\cH)
\sim\binom{y}k\qquad(y\to\infty),
\ee
which asserts that the singular series has an average value of one.
Sharper versions of \eqref{eq:Gallagher} exist
(see, e.g., Montgomery and Soundararajan \cite{MS});
such results, however, are uniform only in a range $|\cH| \ll \log_2 y$ 
or so, far too restrictive for our use.
 Reinterpreting the sum on the left side of
 \eqref{eq:Gallagher} probabilistically, as we have done above, 
 allows us to adequately deal with a much larger range of
 sizes $|\cH|$.
In particular, it is possible to deduce from a uniform version
of \eqref{eq:HL-precise} a uniform version of 
\eqref{eq:cramerexp}, although we have not done so in this paper.

We take this occasion to mention a recent unconditional theorem of 
Mastrostefano \cite[Theorem 1.1]{Mastrostefano},
which is related to
\eqref{eq:Gallagher}, and which states that
for any integer $m\ge 0$ there is an $\eps=\eps(m)>0$
so that whenever $0<\lambda<\eps$, we have
\[
|\{ n\le x : |[n,n+\lambda \log n] \cap \cP|=m \}| \gg_{\lambda,\eps} x.
\]
Establishing the Poisson distribution \eqref{eq:Gallagher-says}
unconditionally, even for some fixed $\lambda$, seems very difficult.

 
\subsection{The maximal gap in Granville's model}

The claimed bounds in Theorem~\ref{thm:conditional} are also satisfied
by Granville's random set $\cG$, i.e., one has
\[
g((\xi-o(1))\log^2 x) \le G_{\cG}(x) \le g((\xi+o(1))\log^2 x).
\]
The proof is very short, and we sketch it here
as a prelude to the proof of Theorem~\ref{thm:conditional}.
Consider the elements of $\cG$ in $(x,2x]$ for $x$ a power of two. 
 In accordance with \eqref{g-bounds}, let $y$ 
satisfy $\log^2 x \le y=o(\log^2 x \log_2 x)$ and
put $A\defeq (y/\log y)^{1/2}$, so that $A=o(\log x)$.
Let $\theta\defeq \prod_{p\le A} (1-1/p)^{-1} \sim (e^\gamma/2)\log y$ and $Q\defeq \prod_{p\le A}p$.
 For simplicity, we suppose that each $n\in (x,2x]$
with $(n,Q)=1$ is chosen for inclusion in $\cG$ with probability
\change{$\theta/\log x$}; this modification 
has a negligible effect on the
size of the largest gap. Fix $\eps>0$ arbitrarily small.
Let $X_m$ denote the event $(m,m+y] \cap \cG=\varnothing$.

\change{Let $D_m$ denote the number of integers in $(m,m+y]$, all of whose
prime factors are larger than $A$.}
If we take $y\defeq g((\xi+\eps)\log^2 x)$, then
\dalign{
\E\big|\{x<m\le 2x: X_m \}\big| &=\sum_{x<m\le 2x} (1-\theta/\log x)^{\change{D_m}} \\
&\le x (1-\theta/\log x)^{W_y} \le x e^{-\theta W_y/\log x}
\\ &\ll x^{-\eps/2}
}
by our assumption that $W_y\log y \sim (\xi+\eps)\log^2 x$.
Summing on $x$ and applying Borel-Cantelli, we see that
almost surely, only finitely many $X_m$ occur.

For the lower bound, we take $y\defeq g((\xi-\eps)\log^2 x)$ and
restrict to special values of $m$,
namely $m\equiv b \bmod Q$, where $b$ is chosen so that
\[
\change{D_b=W_y.}
\]
Let $\cM\defeq \{x<m\le 2x: m\equiv b\bmod Q\}$
and let $N$ be the number of $m\in \cM$ for which $X_m$ occurs.
By the above argument, we see that
\[
\E N=|\cM| (1-\theta/\log x)^{W_y}.
\]
By assumption, $|\cM|=x^{1-o(1)}$ and hence
the right side is $> x^{\eps/2}$ for large $x$.
Similarly,
\dalign{
\E N^2 &=|\cM| (1-\theta/\log x)^{W_y}+(|\cM|^2-|\cM|) (1-\theta/\log x)^{2W_y} \\
&=(\E N)^2+O(\E N).
}
 By Chebyshev's inequality, $\PR(N < \frac12 \E N) \ll
 1/\E N \ll x^{-\eps/2}$.
Considering all $x$ and using Borel-Cantelli,
we conclude that almost surely every sufficiently large
dyadic $(x,2x]$ contains \change{an} $m$ for which $X_m$ occurs.

\change{We remark that our lower bound argument above works as well for the
Cram\'er model, showing \eqref{eq:Cramer-gap}. We take $A=Q=\theta=b=1$,
and the details are simpler.}

\bigskip\textbf{Acknowledgements.}
The authors thank Andrew Granville, Ben Green,
 D.~R.~Heath-Brown, Henryk Iwaniec, Dimitris Koukoulopoulos, James Maynard, Carl Pomerance
and Joni Ter\"av\"ainen for useful discussions,
especially concerning the interval sieve problem.


{\Large\section{Preliminaries}
\label{sec:prelim}}

\medskip


\subsection{Notation}\label{notation-sec}

The indicator function of any set $\cT$ is denoted $\ind{\cT}(n)$. We select residue classes $a_p \bmod p$ uniformly and independently at random for each prime $p$, and then for any set of primes $\cQ$ we denote by $\cA_\cQ$
the ordered tuple $(a_p:p\in\cQ)$; often we condition our
probabilities on $\cA_\cQ$ for a fixed choice of $\cQ$.

Probability, expectation, and variance are denoted by $\P$, $\E$, and $\V$ respectively.
We use $\P_\cQ$ and $\E_\cQ$ to denote
the probability and expectation, respectively, with respect 
to random $\cA_\cQ$. When $\cQ$ is the
set of primes in $(c,d]$, we write 
$\cA_{c,d}$, $\P_{c,d}$ and $\E_{c,d}$; if $\cQ$
is the set of primes $\le c$, we write $\cA_c$, $\P_c$ and $\E_c$.
In particular, $\P_{c,d}$ refers to the probability
over random $\cA_{c,d}$, often with conditioning on $\cA_{c}$.

Throughout the paper, any implied constants in symbols $O$, $\ll$
and $\gg$ are \emph{absolute} (independent of any parameter) unless otherwise indicated.
The notations $F\ll G$, $G\gg F$ and
$F=O(G)$ are all equivalent to the statement that the inequality
$|F|\le c|G|$ holds with some constant $c>0$. We write $F\asymp G$
to indicate that $F\ll G$ and $G\ll F$ both hold.
The notation $o(1)$ is used to indicate a function that tends
to zero as $x\to\infty$; in expressions like $1-o(1)$, the function
is assumed to be positive.
We write \change{$F \sim G$ when $F=(1+o(1)) G$ as $x\to\infty$.}

\change{For a set $\cH$ of integers, we denote
$\cH-\cH \defeq \{h-h': h,h'\in \cH\}$,
and for any integer $m$, $\cH+m \defeq \{ h+m:h\in \cH\}$.}
\medskip


\subsection{Various inequalities}

We collect here some standard inequalities from sieve theory and probability that are used in the rest of the paper.

\begin{lemma}[Upper bound sieve,{\cite[Theorem 3.8]{MV}}]\label{lem:sieve-upper}
For $1\le w\le p\le y$, $p$~prime, $b\in\Z/p\Z$, and an
arbitrary interval $\cI$ of length $y$, we have uniformly
\[
\big|\{n\in\cI: n\equiv b\bmod p,\big(n,\prod_{q\le w} q\big)=1\}\big|
\ll \frac{y/p}{1+\min\{\log w,\log(y/p)\}}.
\]
\end{lemma}

\begin{lemma}[Azuma's inequality \cite{Azuma}]\label{lem:Azuma-ineq}
Suppose that $X_0,\ldots,X_n$ is a martingale with $|X_{j+1}-X_j|\le c_j$
for each $j$. Then
\[
\P\( |X_n-X_0| \ge t \) \le 2 \exp \left\{
- \frac{t^2}{2(c_0^2+\cdots+c_{n-1}^2)}\right\}
\qquad(t>0).
\]
\end{lemma}

\begin{lemma}[Bennett's inequality\cite{Bennett}]\label{lem:Bennett-ineq}
Suppose that $X_1,\ldots,X_n$ are independent random variables
such that for each $j$, $\E X_j=0$, and $|X_j|\le M$ holds
with probability one. Then
\[
\P\bigg(\bigg|\sum_{1\le j\le n} X_j\bigg| \ge t\bigg)
\le 2\exp\bigg\{- \frac{\sigma^2}{M^2}\,\sL\bigg(\frac{Mt}{\sigma^2}\bigg)
\bigg\}\qquad(t>0),
\]
where $\sigma^2\defeq\sum_j \V X_j$, and
\[
\sL(u)\defeq\int_1^{1+u}\log t\,dt=(1+u)\log(1+u)-u.
\]
\end{lemma}

\begin{lemma}\label{lem:SS}
For any $\cH \subset [0,y]$ with $|\cH|=k$, we have
\be\label{SSz}
\fS_z(\cH)=\fS(\cH) \bigg(1+O\bigg( \frac{k^2}{z} \bigg) \bigg) \qquad \change{(z>\max(y,k^2))}
\ee
and
\be\label{SS-upper}
\change{\fS(\cH) \le e^{O(k\log_2(y))}.}
\ee
\end{lemma}

\begin{proof}
Estimate \eqref{SSz} follows from the definition of $\fS(\cH)$
and the fact that for $p>y$, $|\cH \bmod p|=k$.
Estimate \eqref{SS-upper} is \change{trivial if $\cH$ is inadmissible, since
then $\fS(\cH)=0$, and otherwise \eqref{SS-upper} is} a special case of \cite[(6.16)]{GPY}.
\end{proof}

\change{
\begin{lemma}\label{lem:VH}
If $\cH \subseteq [0,y]$ is an admissible $k$-tuple and $t\ge 2$ satisfies
$z(t) > y$ and $k\le t^{1/100}$, then
\[
\TrunSS{z(t)}=\frac{\fS(\cH)}{(\log t)^k} \big( 1+O(1/t^{0.55}) \big).
\]
\end{lemma}
\begin{proof}
Let $z\defeq z(t)$.
By \eqref{eq:theta}, $z \gg t^{1/e^{\gamma}} \gg t^{0.561}$.
Using Lemma~\ref{lem:SS} and \eqref{eq:theta}, we have
\begin{align*}
\TrunSS{z(t)} &=\fS_z(\cH) \Theta_z^{k} \\
&=\fS(\cH) \(1+O\pfrac{k^2}{z}\) \( \frac{1}{\log t}+O(t^{-1/e^\gamma}) \)^k.
\end{align*}
The lemma now follows since $k\le t^{1/100}$.
\end{proof}
}

{\Large\section{Uniform Hardy-Littlewood from the model}
\label{sec:UHL_model}}

In this section, we prove Theorems~\ref{thm:UHL} and~\ref{thm:RHB}
using the first and second moment bounds provided by the following
proposition.

\bigskip

\begin{proposition}[First and second moment bounds]\label{prop:variance}
Suppose that \change{$x$ and $y$ are integers with $x\ge 3$ and $\sqrt{x}\le y\le x$, and suppose that $0\le D\le \sqrt{x}$.
Let $\cH\subset [0,D]$ be an admissible tuple} with
$k\defeq|\cH|\le\frac{\log x}{(\log_2 x)^2}$, and put
\[
X_n \defeq \prod_{h\in \cH} \ind{\cR}(n+h)\qquad(n\in\NN).
\] Then
\be\label{eq:1st-mom}
\E\bigg(\sum_{x<n\le x+y} X_n\bigg)=\change{ \fS(\cH) \int_x^{x+y} \frac{dt}{(\log t)^{k}}+O\bigg( \frac{yD}{x}+\frac{y}{x^{0.54}} \bigg).}
\ee
Furthermore,
\be\label{eq:2nd-mom}
\mathbb{V} \bigg( \sum_{x<n\le x+y} X_n \bigg) \ll y \(
\frac{D}{x} \change{+\frac{yD^2}{x^2}}+\TrunSS{z(x)}\change{(k^2+yD/x)}+\TrunSS{z(x)}^2 F \),
\ee
\change{where} 
\[
F\defeq\begin{cases}
(\log x)^{k^2}
&\quad\hbox{if $k\le \frac{(\log x)^{1/2}}{\log_2 x}$},\\
y^{\frac{4\varrho^2-1}{4\varrho^2-\varrho}}
\exp\Big\{O\Big(\frac{\log x\log_3 x}{\log_2 x}\Big)\Big\}
&\quad\hbox{if $\frac{(\log x)^{1/2}}{\log_2 x}
\le k=(\log x)^\varrho\le \frac{\log x}{(\log_2 x)^2}$}.
\end{cases}
\]
\end{proposition}

\medskip

Before turning to \change{the proof} of the proposition,
we first indicate how it is used to
prove the two theorems, starting with Theorem~\ref{thm:RHB}.

\medskip

\begin{proof}[Proof of Theorem~\ref{thm:RHB}]
\change{Fix $c>3/2$.}
\change{For any integers $u\ge 2$ and $v\ge 0$, we let
\[
\Delta(u,u+v)\defeq\sum_{u<n\le u+v}\ind{\cR}(n)-
\int_u^{u+v}\frac{dt}{\log t}.
\]
We apply Proposition~\ref{prop:variance} in the case that
$\cH=\{0\}$, $k= 1$ and $D= 0$.
By \eqref{eq:1st-mom}, if $v\ge \sqrt{u}$ then
\be\label{eq:Deltauv}
\E \Delta(u,u+v) \ll \frac{v}{u^{0.54}} \ll u^{0.46}.
\ee
 Inequality \eqref{eq:2nd-mom} implies that
\[
\V\big(\Delta(u,u+v)\big)\ll v \big(\TrunSS{z(u)}+\TrunSS{z(u)}^2\log u \big)\ll \frac{v}{\log u}.
\]
Let $x$ be a large integer.
For integers $h,m$ with $2\sqrt{x} \le 2^m \le x$
and $0\le h\le x/2^m-1$, let $G_{m,h}$ be the event that
\[
\big|\Delta(x+h\cdot 2^m,x+(h+1)2^m)\big|
\le x^{1/2}(\log x)^{c-1}.
\]
For large $x$, \eqref{eq:Deltauv} implies that
\[
\big| \E \Delta(x+h\cdot 2^m,x+(h+1)2^m)\big| \le \frac{x^{1/2} (\log x)^{c-1}}{2}.
\]
Hence, Chebyshev's inequality yields the bound
\[
\P\big( \text{not } G_{h,m} \big)
\ll\frac{2^m}{x(\log x)^{2c-1}}.
\]
Let $F_x$ denote the event that $G_{h,m}$ holds for all such $h,m$.
By a union bound, we see that $\P F_x=1-O((\log x)^{2-2c})$. }
On this event \change{$F_x$},
for any \change{integer $y$ with} $1\le y\le x$, we have
\begin{align*}
|\Delta(x,x+y)| &=\bigg| \sum_{\change{2\sqrt{x}}\le 2^m \le y} \Delta\Big( x+\fl{y/2^{m+1}}2^{m+1}, x+\fl{y/2^m} 2^m \Big) \bigg| \change{+O(\sqrt{x})} \\
&\le \sum_{\change{2\sqrt{x}}\le 2^m \le y} x^{1/2} (\log x)^{c-1} 
\change{+O(\sqrt{x})} \\ &\change{\ll} x^{1/2} (\log x)^c.
\end{align*}

\change{Since $2c-2>1$, the Borel-Cantelli lemma implies that with probability one,
$F_{2^s}$ is true for all large integers $s$. On this event,
$\Delta(2,x) \ll x^{1/2} (\log x)^c$ for all real $x\ge 2$, proving the theorem.}
\end{proof}

\medskip

\begin{proof}[Proof of Theorem~\ref{thm:UHL}]

\change{
Fix $c\in[1/2,1)$ and $\eps>0$. For integers $a\ge 2$, $b\ge 0$ and a tuple $\cH$, define
\[
\Delta(a,a+b;\cH) \defeq \sum_{a<n\le a+b} \prod_{h\in \cH} \ind{\cR}(n+h) -
\fS(\cH) \int_a^{a+b} \frac{dt}{(\log t)^{|\cH|}}.
\]
Let 
\[
\lambda \defeq 1 - \frac{1-c}{8c^2-2c}.
\]
}

\change{Let $u$ be a large integer in terms of $c$ and $\eps$, and
let $F_u$ denote the event that
\[
|\Delta(a,a+b;\cH)| \le u^{\lambda+\eps}
\]
for all integers $a,b$ satisfying $u\le a \le a+b \le 2u$ and all 
admissible tuples $\cH$ satisfying
\be\label{H-range}
|\cH|=k\le 10(\log u)^c, \quad \cH \subset \Big[ 0,
\exp\Big\{ 10(\log u)^{1-c}/\log_2 u \big\} \Big].
\ee
The number of such $\cH$ does not exceed $u^{100/\log_2 u}=u^{o(1)}$ as
$u\to\infty$.
}

\change{We again invoke the moment bounds in Proposition \ref{prop:variance}.
Assume $\cH$ satisfies \eqref{H-range} and that $u\le a\le 2u$ and $\sqrt{a}\le b\le a$.
It follows from \eqref{eq:1st-mom} that
\[
\E \Delta(a,a+b;\cH) \ll \frac{b u^{o(1)}}{a}+\frac{b}{a^{0.54}} \ll u^{0.46},
\]
and inequality \eqref{eq:2nd-mom} implies
\[
\change{\mathbb{V} \Delta(a,a+b;\cH)} \ll b^{1+\frac{4c^2-1}{4c^2-c}+o(1)} a^{o(1)}
\ll b u^{2\lambda-1+o(1)},
\]
where the implied function $o(1)$ is uniform over all such $\cH$, $a$ and $b$.
For integers $h,m$ with $2\sqrt{u} \le 2^m \le u$ and $0\le h \le u/2^m-1$,
let $G_{h,m}$ be the event that for all $\cH$ satisfying \eqref{H-range},
\[
|\Delta(u+h\cdot 2^m, u+(h+1)\cdot 2^m;\cH)| \le u^{\lambda+\eps/2}. 
\]
Again, if $u$ is large enough, the expectation of the left side is
at most $\frac12 u^{\lambda+\eps/2}$, uniformly over all $h,m,\cH$.
By a union bound and Chebyshev's inequality,
\begin{align*}
\PR \big( \cup_{h,m} (\text{not } G_{h,m})\big) &\le \sum_{h,m} \sum_{\cH} 
\P\big(\big|\Delta(u+h\cdot 2^m,u+(h+1)\cdot 2^m;\cH)\big|\ge \tfrac12 u^{\lambda+\eps/2}\big)
\\ &\ll \sum_{h,m} \sum_{\cH} \frac{2^m}{u^{1+\eps+o(1)}}
\ll \frac{1}{u^{\eps/2}}.
\end{align*}
Furthermore, as in the proof of Theorem~\ref{thm:RHB}, we see that if $u$ is
large enough (in terms of $c,\eps$) and if $G_{h,m}$ holds for all $h,m$,
then $F_u$ holds. Therefore, 
\[
\PR F_u=1 -O\big( 1/u^{\eps/2} \big).
\]
By Borel-Cantelli, almost surely $F_{2^s}$ is true for all sufficiently large
integers $s$. }

\change{
Now assume that we are in the event that $F_{2^s}$ holds for all $s\ge s_0$.
Let $x$ be sufficiently large such that $x \ge 2^{3s_0+1}$ and \red{$2^{s_1} < x \le 2^{s_1+1}$}, and let $\cH$ be an admissible tuple
with
\[
k \defeq |\cH| \le (\log x)^c, \qquad \cH \subseteq \Big[ 0, \exp \Big\{
\frac{(\log x)^{1-c}}{\log_2 x} \Big\}\Big].
\]
Note that whenever $x^{1/3} \le u=2^s \le x$ we have \eqref{H-range}.
Thus, using \eqref{SS-upper},
\begin{align*}
\bigg|\sum_{n\le x} \prod_{h\in \cH} \ind{\cR}(n+h)&
-\fS(\cH)\int_{2}^{x} \frac{dt}{\log^{|\cH|} t}\bigg| \le O(x^{1/3+o(1)})+\\
& \quad+\sum_{x^{1/3} < 2^s \le x/2} |\Delta(2^s,2^{s+1};\cH)|+|\Delta(\red{2^{s_1}},x;\cH)|\\
&\ll x^{\lambda+\eps/2},
\end{align*}
as required for Theorem~\ref{thm:UHL}.
}
\end{proof}

\change{
The following lemma is needed in the proof of Proposition \ref{prop:variance}.
When an admissible tuple $\cH$ is fixed, define
\[
\psi_t \defeq \TrunSS{z(t)}.
\]
\begin{lemma}\label{lem:psit}
Let $2\le u \le v \le 3u$, and suppose $\cH$ is an admissible tuple with
$k\defeq |\cH| \ge 1$. Then
\[
\psi_u - \psi_v \ll k\psi_u \( \frac{1}{u^{1/e^\gamma}}+
\frac{v-u}{u\log u} \).
\]
\end{lemma}
\begin{proof}
We begin with the simple bound
\be\label{eq:psiuv}
\begin{split}
\psi_u-\psi_v &=\psi_u\bigg(1-\prod_{z(u)<p\le z(v)}
(1-\nu_p/p)\bigg) \\
&\le \psi_u\sum_{z(u)<p\le z(v)}\frac{\nu(p)}{p}\\
&\le k\psi_u\sum_{z(u)<p\le z(v)}\frac{1}{p}.
\end{split}\ee
By multiple applications of \eqref{eq:theta},
\begin{align*}
\sum_{z(u)<p\le z(v)}\frac{1}{p} &\le \sum_{z(u)<p\le z(v)} -\log(1-1/p)=
\log\bigg( \Theta_{z(u)}/\Theta_{z(v)} \bigg) \\
&=\log \( \frac{\log v}{\log u}\Big(1+O(1/z(u)) \Big) \)\\
&\ll \frac{1}{z(u)}+\log \( 1+\frac{\log(v/u)}{\log u} \) \\
&\ll \frac{1}{z(u)}+\frac{\log(v/u)}{\log u}\\
&\ll \frac{1}{u^{1/e^\gamma}}+\frac{v-u}{u\log u}.
\end{align*}
This completes the proof.
\end{proof}
}

\medskip

\begin{proof}[Proof of Proposition \ref{prop:variance}]
Suppose that \change{$\cH\subset [0,D]$} with 
$k\defeq|\cH| \le \frac{\log x}{(\log_2 x)^2}$. \change{We may assume
that $D$ is an integer.
Write $\nu_p\defeq|\cH\bmod p|$ for every prime $p$.
Since $z(t)$ is increasing and $\psi_u$ is decreasing in $u$,
\[
\psi_{n+D} \le \E X_n \le \psi_n.
\]
Hence,
\[
\E \sum_{x<n\le x+y} X_n=\sum_{x<n\le x+y} \psi_n+O\bigg( \sum_{j=1}^D
\big( \psi_{x+j} - \psi_{x+y+j} \big) \bigg).
\]
By Lemma~\ref{lem:psit} and the bound $\psi_u\ll 1/\log u$, the big-$O$ term is
\[
\ll \frac{kD}{\log x} \bigg( \frac{1}{x^{1/e^\gamma}}+\frac{y}{x\log x} \bigg)
\ll \frac{kDy}{x\log^2 x},
\]
since $y\ge \sqrt{x}$ and $1/e^{\gamma} > 1/2$.
This proves that
\be\label{eq:E1}
\E \sum_{x<n\le x+y} X_n=\sum_{x<n\le x+y} \psi_n+O\pfrac{kDy}{x\log^2 x}.
\ee
Lemma~\ref{lem:VH} implies that for each integer $n\in (x,x+y]$ we have
\[
\psi_n=\frac{\fS(\cH)}{(\log n)^k}\(1+O(1/x^{0.55}) \)
=\fS(\cH) \int_{n-1}^n \frac{dt}{(\log t)^k}+O\pfrac{\fS(\cH)}{x^{0.55}}.
\]
Estimate \eqref{SS-upper} implies that $\fS(\cH) \le x^{o(1)}$ and
this proved the estimate \eqref{eq:1st-mom} of the proposition.
}

For the second moment bound, let $v$ be a parameter in
$[4k,\log x]$ and set $Q\defeq\prod_{p\le v} p$.
Given integers $n_1$ and $n_2$ with $x<n_1 < n_2 \le x+y$, define
$m$ and $b$ by
\[
m\defeq n_2-n_1,\qquad
b\equiv m\bmod Q\quad\text{with}\quad b\in[0,Q).
\]
We consider separately the primes $\le v$ and those $>v$, setting
\[
\psi'_n\defeq\prod_{v<p\le z(n)} \(1-\frac{\nu_p}{p}\), \qquad
\xi_b\defeq\prod_{p\le v}\(1-\frac{|(\cH\cup (\cH+b))\bmod p|}{p}\).
\]
Then
\begin{equation}
\label{EX1X2-1}
\begin{split}
\E X_{n_1}X_{n_2}&\le\prod_{p\le z(n_1)}\(1-\frac{|(\cH\cup(\cH+m))
\bmod p|}{p}\)\prod_{z(n_1)<p\le z(n_2)}\(1-\frac{\nu_p}{p}\)\\
&=\frac{\psi'_{n_2}}{\psi'_{n_1}}\,\xi_b
\prod_{v<p\le z(n_1)}\(1-\frac{|(\cH\cup(\cH+m))\bmod p|}{p}\).
\end{split}
\end{equation}
For technical reasons, we use the trivial bound
$\E X_{n_1}X_{n_2}\le\psi_{n_1}\le\psi_x$
when $m\in\cH-\cH$; the total contribution from such
terms is $\le \psi_xk^2 y$, which is
an acceptable error term for \eqref{eq:2nd-mom}.

Now suppose that $m\not\in\cH-\cH$. For any \change{prime
$p>v$} and integer $a\in(-p/2,p/2)$, let
\[
\lambda_a(p)\defeq |(\cH \cap (\cH+a))\bmod p|.
\]
Then, given $v<p\le z(x+y)$ and $m$ we have
\[
|(\cH\cup (\cH+m))\bmod p|=2\nu_p-\lambda_a(p),
\]
where $a$ is the unique integer such that
\[
a\equiv m\bmod p\quad\text{and}\quad \change{|a|<p/2}.
\]
Clearly, $\lambda_a(p) \le \nu_p \le k$, and 
$\lambda_a(p)=0$ unless $a\in(\cH-\cH)\cap(-p/2,p/2)$.
In addition,
\be\label{sumlam}
\sum_a \lambda_a(p)=\nu_p^2.
\ee
Consequently, for any $p>v$ we have
\[
1- \frac{|(\cH\cup (\cH+m))\bmod p|}{p}=
\bigg(1-\frac{2\nu_p}p\bigg)(1+f_a(p)) 
\]
with
\[
f_a(p)\defeq\frac{\lambda_a(p)}{p-2\nu_p}
\]
We remark that $f_a(p)\in(0,1]$ since $p>v\ge 4k\ge 4\nu_p \change{\ge 4\lambda_a(p)}$.
For a fixed choice of $a\in \cH - \cH$ and \change{fixed $n_1$}, 
extend $f_a$ to a multiplicative
function supported on squarefree integers whose prime factors all
lie in \change{$I(n_1,a)\defeq (\max\{v,2|a|\},z(n_1)]$.}
\change{If an integer $r$ has a prime factor outside the interval $I(n_1,a)$ or
$r$ is not squarefree, we set $f_a(r)\defeq 0$.}
 Then
\begin{align*}
&\prod_{v<p\le z(n_1)}\(1-\frac{|(\cH\cup(\cH+m))\bmod p|}{p}\)\\
&\qquad\qquad=\prod_{v<p\le z(n_1)}\(1-\frac{2\nu_p}p\)
\prod_{a\in\cH-\cH}~\sprod{v<p\le z(n_1)\\p\,\mid\,m-a}(1+f_a(p))\\
&\qquad\qquad=\prod_{v<p\le z(n_1)}\(1-\frac{2\nu_p}p\)
\prod_{a\in\cH-\cH}~\sum_{d_a\,\mid\,(m-a)}f_a(d_a)
\end{align*}
(since $m\not\in \cH-\cH$, we always have $m-a\ne 0$).
Recalling \eqref{EX1X2-1} we obtain that
\be\label{eq:EX1X2-2}
\E X_{n_1} X_{n_2}\le \psi'_{n_1}\psi'_{n_2}\xi_b 
\prod_{v<p\le z(n_1)} \pfrac{p^2-2p\nu_p}{(p-\nu_p)^2} S(n_1,n_2),
\ee
where 
\[
S(n_1,n_2)\defeq\prod_{a\in\cH-\cH}~\sum_{d_a\,\mid\,(m-a)} f_a(d_a).
\]

We now fix $n_1$ and sum over $n_2$. Let
\change{\begin{align*}
\cD(n_1)\defeq\big\{\mathbf{d}&=(d_a)_{a\in\cH-\cH}:
\exists\,m\in [1,y]\setminus(\cH-\cH)\text{~such that~}
\forall\,a,\;~d_a\mid(m-a),\\
&\text{ each } d_a \text{ is squarefree with all of its prime
factors in } I(n_1,a)\big\},
\end{align*}}
i.e., $\cD(n_1)$ is the set of all possible vectors of the
numbers $d_a$. We compute
\begin{align*}
\ssum{n_1<n_2\le x+y\\n_2-n_1\not\in\cH-\cH}\psi'_{n_2}\xi_b\,S(n_1,n_2)
\le\sum_{\mathbf{d}\in\cD(n_1)}\Big(\prod_af_a(d_a)\Big)
\sum_{b\bmod Q}\xi_b\ssum{n_1<n_2\le x+y\\n_2\equiv n_1+b\bmod Q\\
\forall a,\;n_2\equiv n_1+a\bmod{d_a}}\psi'_{n_2},
\end{align*}
\change{where we have dropped the condition $n_2-n_1\not\in \cH-\cH$ on the
right side.}
A crucial observation is that for every $\mathbf{d}\in \cD(n_1)$,
 the components $d_a$ are pairwise coprime. 
\change{Indeed, if $a,a'$ are two distinct elements of $\cH-\cH$
and a prime $p>\max\{v,2|a|,2|a'|\}$ divides both $d_a$ and $d_{a'}$, then 
there is some $m\in [1,y]\setminus (\cH-\cH)$
so that $p\,\mid\,d_a\,\mid\,(m-a)$ and
$p\,\mid\,d_{a'}\,\mid\,(m-a')$. This implies $a\equiv a'\pmod{p}$, a contradiction.}
 Hence, the innermost sum is a sum over a single residue class modulo $d\defeq Q\prod_a d_a$.
For any $e\in \Z$ we have by \change{\eqref{eq:psiuv}} that
\begin{align*}
\ssum{n_1<n\le x+y \\ n\equiv e\bmod d} \psi'_{n} &=
\ssum{n_1<n\le x+y \\ n\equiv e\bmod d} \bigg[
\frac{1}{d}(\psi'_n+\cdots+\psi'_{n+d-1})+O\bigg( 
k\psi'_x \sum_{z(n)<p\le z(n+d)} \frac{1}{p} \bigg) \bigg] \\
&=O(\psi'_x)+\frac{1}{d} \sum_{n_1<n\le x+y} \psi'_{n},
\end{align*}
\change{where we used that $k\le \log x$ and 
\[
\sum_{z(x) <p \le z(x+y+d)} \frac{1}{p} \ll \frac{1}{\log x}.
\]
}
Therefore,
\be
\label{eq:SunSn1n2}
\begin{split}
\ssum{n_1<n_2\le x+y\\n_2-n_1\not\in\cH-\cH}
\psi'_{n_2}\xi_b\,S(n_1,n_2)
&\le 
\frac{1}{Q} \sum_{b\bmod Q} \xi_b 
\ssum{\change{n_1<n_2\le x+y}} \psi'_{n_2}
\sum_{\mathbf{d}\in \cD(n_1)} \prod_a \frac{f_a(d_a)}{d_a}\\
&\qquad+O\bigg( \psi'_x \sum_{b\bmod Q} \xi_b
\sum_{\mathbf{d}\in \cD(n_1)} \prod_a f_a(d_a) \bigg).
\end{split}
\ee
Now \eqref{sumlam} implies that
\dalign{
\sum_{b\bmod Q} \xi_b &=\prod_{p\le v} \sum_{c=0}^{p-1}
\(1-\frac{|(\cH \cup (\cH+c))\bmod p|}{p}\) \\
&=\prod_{p\le v} \bigg( p - 2\nu_p+\frac{1}{p} \sum_a \lambda_a(p) \bigg) \\
&=Q \prod_{p\le v}\(1-\frac{\nu_p}{p}\)^2.
}
Hence, combining \eqref{eq:EX1X2-2} and \eqref{eq:SunSn1n2},
and reinserting terms with $n_2-n_1\in \cH-\cH$, for each $n_1$
we obtain that
\dalign{
\E\sum_{n_1<n_2\le x+y} X_{n_1} X_{n_2}
&\le \psi_{n_1}
\sum_{n_1<n_2\le x+y} \psi_{n_2} \prod_{v<p\le z(n_1)} \pfrac{p^2-2p\nu_p}{(p-\nu_p)^2} \sum_{\mathbf{d}\in \cD(n_1)} \prod_a \frac{f_a(d_a)}{d_a} \\
&\qquad+O\Bigg( \psi_x^2 Q \sum_{\mathbf{d}\in \cD(n_1)} \prod_a f_a(d_a)+\psi_x k^2 \Bigg).
}
Extending the first sum over $\mathbf{d}$ to all pairwise coprime tuples $\mathbf{d}$ composed of prime factors in $(v,z(n_1)]$,
and applying \eqref{sumlam} again,
we find that
\dalign{
\sum_{\mathbf{d}\in\cD(n_1)}\prod_a \frac{f_a(d_a)}{d_a} &\le 
\prod_{v<p\le z(n_1)}\(1+\sum_a \frac{f_a(p)}{p}\) \\ &=
\prod_{v<p\le z(n_1)}\(1+\frac{\nu_p^2}{p(p-2\nu_p)}\).
}
Finally, summing over $n_1$ we conclude that
\begin{align*}
\E \sum_{x<n_1<n_2\le x+y} X_{n_1} X_{n_2} &\le 
\sum_{x<n_1<n_2 \le x+y} \psi_{n_1} \psi_{n_2}+
O(\psi_x k^2y+\psi_x^2QTy),
\end{align*}
where
\[
T\defeq\max_{n_1} \sum_{\mathbf{d}\in \cD(n_1)} \prod_a f_a(d_a) .
\]
\change{Since $X_n^2=X_n$ we arrive at
\[
\E \bigg( \sum_{x<n\le x+y} X_n \bigg)^2 \le \E \sum_{x<n\le x+y} X_n+
\ssum{x<n_1,n_2 \le x+y \\ n_1\ne n_2} \psi_{n_1} \psi_{n_2}+
O(\psi_x k^2y+\psi_x^2QTy),
\]
}
Comparing this with \eqref{eq:E1}, it follows that the variance
in question satisfies
\change{\be\label{eq:Var-1}
\begin{split}
\mathbb{V} \sum_{x<n\le x+y} X_n &\le
 \sum_{x<n\le x+y} \big( \psi_n - \psi_n^2 \big)+O\big( \psi_x k^2y+\psi_x^2QTy \big)+\\ &\qquad\qquad\qquad+O\bigg(\frac{yD}{x}\sum_{x<n\le x+y} \psi_n+\frac{y^2D^2}{x^2}+\frac{yD}{x}\bigg) \\
 &\ll y \psi_x+k^2 y \psi_x+\psi_x^2 QTy+\frac{y^2 D}{x}\psi_x+\frac{y^2D^2}{x^2}+\frac{yD}{x}\\
 &\ll k^2 y \psi_x+\psi_x^2 QTy+\frac{yD}{x}\Big[ 1+y(\psi_x+D/x)\Big].
\end{split}\ee
}

To bound $T$, we consider two cases. First, suppose that
$k\le (\log x)^{1/2} / \log_2 x$, and let
$v\defeq 4k$.
\change{In this case, we argue crudely, using \eqref{sumlam} and $\nu_p\le k$ for all $p$,
obtaining}
\begin{align*}
T&\le \prod_{v<p\le z(2x)} \Big( 1+\sum_{|a|<p/2}f_a(p) \Big)\\
&=\prod_{4k<p\le z(2x)}\(1+\frac{k^2}{p-2k}\)\\
&\le\exp\big(k^2(\log_2x-\log_2 k+O(1))\big)
\ll e^{-k^2}(\log x)^{k^2}.
\end{align*}
The prime number theorem implies that $\log Q \ll v$ and
 thus $QT \ll (\log x)^{k^2}$.
Therefore, \eqref{eq:Var-1} implies \eqref{eq:2nd-mom}.

Next, suppose that 
\be\label{eq:large-k}
\frac{(\log x)^{1/2}}{\log_2 x} \le k\le \frac{\log x}{(\log_2 x)^2},
\qquad\text{with}\quad k=(\log x)^{\varrho},
\ee
and put
\be\label{eq:v-def} 
v\defeq\frac{4\log x}{\log_2 x},
\ee
so that $v\ge 4k$ \change{and $Q=x^{o(1)}$}. For a parameter $U\le x^5$, to be chosen later, let
\begin{align*}
\cD^-_U&\defeq\big\{\mathbf{d}\in\cD(n_1):
\textstyle\prod d_a\le U\big\},\\
\cD^+_U&\defeq\big\{\mathbf{d}\in\cD(n_1):
\textstyle\prod d_a>U\big\}.
\end{align*}
We begin \change{with} $\cD_U^-$. For any parameter $\alpha > 0$ we have, by \eqref{sumlam},
\begin{align*}
\sum_{\mathbf{d}\in\cD_U^-}\prod_af_a(d_a) &\le U^{\alpha}
\sum_{\mathbf{d}\in\cD_U^-}\prod_a \frac{f_a(d_a)}{d_a^\alpha}\\ 
&\le U^{\alpha}
\prod_{v<p\le z(2x)}\(1+\frac{k^2}{p^\alpha(p-2k)}\)\\
&\le U^{\alpha}\exp\bigg\{2k^2
\sum_{v<p\le z(2x)}\frac{1}{p^{1+\alpha}}\bigg\}
\\ &\le U^{\alpha}\exp\bigg\{
O\bigg(\frac{k^2}{\alpha v^\alpha\log v}\bigg)\bigg\}.
\end{align*}
Let
\[
\alpha\defeq 2\varrho-1+\frac{3\log_3x}{\log_2x},
\]
so that \change{$\frac{\log_3 x}{\log_2 x} \le \alpha \le 1$} by \eqref{eq:large-k}.
Recalling \eqref{eq:v-def}, we see that
\[
\alpha v^\alpha \log v
\gg \alpha (\log_2 x)^{1-\alpha} (\log x)^{\alpha}
\gg (\log x)^{\alpha}
=k^2(\log_2 x)^3/\log x,
\]
hence it follows that
\be\label{eq:D-minus}
\sum_{\mathbf{d}\in \cD_U^-} \prod_a f_a(d_a) \le U^{2\varrho-1}
\exp \bigg\{ O\bigg( \frac{\log x\log_3 x}{\log_2 x} \bigg) \bigg\}.
\ee

Next, we turn to $\cD_U^+$, and make use of the special 
structure of $\cD(n_1)$.
For any parameter $\beta\in[0,1)$ we have
\begin{align*}
\sum_{\mathbf{d}\in\cD_U^+}\prod_af_a(d_a)&\le U^{-\beta}
\sum_{\mathbf{d}\in\cD(n_1)}\prod_a(f_a(d_a)d_a^\beta)\\
&\le U^{-\beta}\ssum{1\le m\le y\\m\not\in\cH-\cH}
~\ssum{\mathbf{d}\in\cD(n_1)\\\forall a,\;d_a\,\mid\,(m-a)}
\prod_a(f_a(d_a)d_a^\beta)\\
&\le U^{-\beta}\ssum{1\le m\le y\\m\not\in\cH-\cH}
\prod_{a\in\cH-\cH}~\sprod{p\,\mid\,m-a\\\max\{v,2|a|\}<p\le z(2x)}
\bigg(1+\frac{\lambda_a(p)p^{\beta}}{p-2\nu_p}\bigg).
\end{align*}
Note that each prime $p$ can appear at most once in the double product,
since $p\mid(m-a)$ and $p\mid(m-a')$ implies $p\mid(a-a')$, which 
forces $a=a'$. We split the last product into two pieces according to whether
$p\le w$ or $p>w$, where $w$ is a parameter to be chosen later.
For any $m\not\in \cH-\cH$ we have
\begin{align*}
\prod_{a\in\cH-\cH}~\sprod{p\,\mid\,m-a\\\max\{v,2|a|\}<p\le w}
\hskip-10pt\bigg(1+\frac{\lambda_a(p) p^{\beta}}{p-2\nu_p}\bigg)
&\le \prod_{v<p\le w} \(1+2k p^{\beta-1}\) \\
&\le \exp \big\{ 2k w^\beta \log_2 x \big\}
\end{align*}
\change{for large $x$.}
We bound the contribution of larger primes trivially using the
fact that any integer $m-a$ is divisible by $\ll\frac{\log x}{\log_2 x}$
such primes (here, it is crucial that $m\ne a$).
Thus, for any $m\not\in \cH-\cH$ we have
\[
\prod_{a\in\cH-\cH}~\sprod{p\,\mid\,m-a\\\max\{w,2|a|\}<p\le z(2x)}
\hskip-10pt\bigg(1+\frac{\lambda_a(p)p^{\beta}}{p-2\nu_p}\bigg)
\le \exp\bigg\{O\bigg(k^3 w^{\beta-1} \frac{\log x}{\log_2 x}\bigg)\bigg\}.
\]
We now put
\[
w\defeq k^2 \log x\mand
\beta\defeq \frac{1-\varrho - 2 \frac{\log_3 x}{\log_2 x}}{2\varrho+1}.
\]
By \eqref{eq:large-k} we have $\beta \ge 0$, and clearly $\beta<1$.
It follows that
\be\label{eq:D-plus}
\sum_{\mathbf{d}\in \cD_U^+} \prod_a f_a(d_a) \le 
y U^{- \frac{1-\varrho}{2\varrho+1}}
\exp \bigg\{ O\Big( \frac{\log x\,\log_3 x}{\log_2 x} \Big) \bigg\}.
\ee
Comparing \eqref{eq:D-minus} with \eqref{eq:D-plus}, we choose $U$
so that $U^{2\varrho-1}=yU^{-\frac{1-\varrho}{2\varrho+1}}$, that is,
\[
U\defeq y^{\frac{2\varrho+1}{4\varrho^2-\varrho}}.
\]
\change{Since $1/2+o(1) \le \rho \le 1+o(1)$, the exponent of $y$ is $\le 4+o(1) \le 5$
for large $x$.}
This gives
\[
T \le 
y^{\frac{4\varrho^2-1}{4\varrho^2-\varrho}}
\exp \bigg\{ O\Big( \frac{\log x\log_3 x}{\log_2 x} \Big) \bigg\}.
\]
Inserting this into \eqref{eq:Var-1} yields the inequality
\eqref{eq:2nd-mom}, and completes the proof of Proposition \ref{prop:variance}.
\end{proof}


{\Large\section{Random sieving by small primes}
\label{sec:gaps-randomly-sieved}}

Throughout the sequel, we employ the notation
\be\label{eq:Theta-notations}
\Theta_z\defeq\prod_{p\le z}\bigg(1-\frac1p\bigg)\mand
\Theta_{z_1,z_2}\defeq\prod_{z_1<p\le z_2}\bigg(1-\frac1p\bigg)
=\frac{\Theta_{z_2}}{\Theta_{z_1}}.
\ee

Throughout this section,
we assume that $x$ and $y$ are large real numbers that satisfy
\be\label{eq:xy}
W_y\log y\in[\alpha(\log x)^2,\beta(\log x)^2],
\ee
where $W_y$ is given by \eqref{eq:By-defn}, and $\alpha,\beta$ are fixed
with $0<\alpha<\beta$. Note that \eqref{eq:By-bounds} and \eqref{eq:xy}
yield the estimates
\be\label{eq:yx-relns}
(\log x)^2\ll y\ll\frac{\log_2x}{\log_3x}(\log x)^2.
\ee
We adopt the convention that any constants
implied by $O$ and $\ll$ may depend on $\alpha,\beta$ but are
independent of other parameters.

We define
\[
\cS_w(y) \defeq [0,y]\cap\cS_w
\]
and when the value of $y$ is clear from context we put
\[
S_w\defeq|\cS_w(y)|.
\]
Using a variety of tools,
we give sharp probability bounds for $S_w$ at five
different ``checkpoint'' values $w_1<w_2<w_3<w_4<w_5$ (defined below), with each $S_{w_{i+1}}$ controlled in terms of $S_{w_i}$ for $i=1,2,3,4$. Our arguments are summarised as follows, \change{where the range is a range of primes:}

\bigskip
 
\begin{tabular}{rl}
\hline
\change{Range} & Estimation technique \\
\hline
\change{$[2,w_1]$} & Lower bound by $W_y$ \eqref{easy} \\
\change{$(w_1,w_2]$} & Buchstab identity, sieve upper bound (Lemma~\ref{lem:w1w2}) \\
\change{$(w_2,w_3]$} & Buchstab identity, large sieve, Bennett inequality (Lemma~\ref{lem:w2w3})\\
\change{$(w_3,w_4]$} & Martingale interpretation, Azuma inequality (Lemma~\ref{lem:w3w4})\\
\change{$(w_4,w_5]$} & \change{Graph} interpretation, combinatorial expansion (Lemma~\ref{lem:w4w5})\\
\change{$(w_5,z]$} & Combinatorial expansion (Lemmas \ref{lem:SzSP}, \ref{lem:VarSP}, Corollary \ref{cor:binomSk}) \\
\hline
\end{tabular}

\bigskip

The most delicate part of the argument is dealing with
primes $p$ near $\log x$, that is, $w_1 \le p\le w_3$
(see Lemmas~\ref{lem:w1w2} and~\ref{lem:w2w3}). To initialize the argument, we observe from definition \eqref{eq:By-defn} of $W_y$ that we have
the lower bound
\be\label{easy}
S_{w_1} \geq W_y.
\ee
Now we successively increase the sieving range from $S_{w_1}$ to $S_{w_2}$, and so on, up to $S_{w_5}$.

\bigskip

\begin{lemma}[Sieving for $w_1 < p \leq w_2$]\label{lem:w1w2}
Let $w_1\defeq(y/\log y)^{1/2}$ and
$w_2\defeq\log x\,\log_3x$.
With probability one, we have
\[
S_{w_2}=\bigg(1+O\bigg(\frac{\log_4 x}{\log_3 x}\bigg)\bigg)S_{w_1}.
\]
\end{lemma}

\begin{proof}
\change{In this section and the next one, we adopt the notation
$R_p$ for the residue class $a_p\bmod p$.}
From the Buchstab identity
$$ S_{w_2}=S_{w_1} - \sum_{w_1 < p \le w_2} |\cS_{p-1}(y)\cap R_p|$$
we have 
\be\label{eq:phone1}
S_{w_1}\ge S_{w_2}\ge S_{w_1}-\sum_{w_1<p\le w_2}|\cS_{w_1}(y)\cap R_p|.
\ee
The sieve upper bound (Lemma~\ref{lem:sieve-upper}) and Mertens' theorem together imply that
\be\label{eq:phone2}
\sum_{w_1<p\le w_2}|\cS_{w_1}(y)\cap R_p|
\ll\frac{y}{\log y}\log\Big(\frac{\log w_2}{\log w_1}\Big)
=S_{w_1}C_y\log\Big(\frac{\log w_2}{\log w_1}\Big),
\ee
where
\[
C_y\defeq\frac{y}{S_{w_1}\log y}.
\]
By \eqref{eq:xy} and \eqref{eq:yx-relns} we have
\be\label{eq:B'ybds}
C_y\le\frac{y}{W_y\log y}\ll\frac{\log_2x}{\log_3x}.
\ee
Using \eqref{eq:xy} and the lower bound $w_1^2=S_{w_1}C_y\ge W_yC_y$
we see that
\[
\log w_1\ge \log_2x-\tfrac12(\log_2y-\log C_y)+O(1),
\]
hence
\begin{align*}
\log\Big(\frac{\log w_2}{\log w_1}\Big)
&\le\log\bigg(\frac{\log_2x+\log_4x}
{\log_2x-\tfrac12(\log_2y-\log C_y)+O(1)}\bigg)\\
&\ll\frac{\log_2y-\log C_y}{\log_2x}\ll\frac{\log_3x-\log C_y}{\log_2x}.
\end{align*}
Inserting this bound into \eqref{eq:phone2} we find that
\[ 
\sum_{w_1<p\le w_2}|\cS_{w_1}(y)\cap R_p|
\ll S_{w_1}\frac{C_y(\log_3x-\log C_y)}{\log_2x}.
\]
The function $z(\log_3 x-\log z)$ is increasing for 
$z\le e^{-1}\log_2 x$, hence by \eqref{eq:B'ybds} we have
\[
\sum_{w_1<p\le w_2}|\cS_{w_1}(y)\cap R_p|
\ll S_{w_1}\frac{\log_4x}{\log_3x}
\]
 and the stated result follows from \eqref{eq:phone1}.
\end{proof}

\begin{lemma}[Sieving for $w_2 < p \leq w_3$]\label{lem:w2w3}
Let $w_2\defeq\log x\,\log_3x$ and
$w_3\defeq\log x\,(\log_2x)^2$.
Conditional on $\cA_{w_2}$ satisfying $S_{w_2} \ge \frac12 W_y$, we have
\[
\P_{w_2,w_3}\bigg(S_{w_3}\le\bigg(1-\frac{1}{\log_3x}\bigg)S_{w_2}\bigg)
\ll x^{-100}.
\]
\end{lemma}

\begin{proof}
As in the previous lemma, we start with
\be\label{eq:Sw2s3}
S_{w_3} \ge S_{w_2} -
\sum_{w_2<p\le w_3} |\cS_{w_2}(y)\cap R_p|.
\ee
Let $X_p\defeq|\cS_{w_2}(y)\cap R_p|-p^{-1}S_{w_2}$
for each prime $p\in (w_2,w_3]$.
The variables $X_p$ are independent and have a mean value of zero,
and by the sieve upper bound (Lemma~\ref{lem:sieve-upper}) it follows that
\[
|X_p|\ll\frac{y}{p\log y}\ll\frac{y}{w_2\log_2x},
\]
hence
\be\label{eq:Bennett1}
|X_p|\le M\defeq\frac{c\,y}{\log x\,\log_2x\,\log_3x}\qquad(w_2<p\le w_3)
\ee
for some absolute constant $c>0$.
Using Montgomery's Large Sieve inequality
(see~\cite[Equation~(9.18)]{FI} or~\cite{Mo}),
\[
\sum_{w_2<p\le w_3}p^2\,\V X_p=\sum_{w_2<p\le w_3}p
\sum_{a\in\Z/p\Z}\bigg(\big|\cS_{w_2}(y)\cap(a\bmod p)\big|
-p^{-1}S_{w_2}\bigg)^2\le 2w_3^2\,S_{w_2},
\]
which implies that
\be\label{eq:Bennett2}
\sigma^2\defeq\sum_{w_2<p\le w_3}\V X_p\le 2w_2^{-2}w_3^2\,S_{w_2}
\ll\frac{(\log_2x)^4}{(\log_3x)^2}S_{w_2}.
\ee
We apply Bennett's inequality (Lemma~\ref{lem:Bennett-ineq}) with
$t\defeq S_{w_2}/(2\log_3x)$. By \eqref{eq:Bennett1}, \eqref{eq:Bennett2} and \eqref{eq:yx-relns},
we have
\[
\frac{Mt}{\sigma^2}\gg\frac{y}{\log x\,(\log_2x)^5}
\gg \frac{\log x}{(\log_2 x)^5},
\]
and therefore
\[
\frac{\sigma^2}{M^2}\sL\Big(\frac{Mt}{\sigma^2}\Big)
\gg\frac{t}{M}\log\Big(\frac{Mt}{\sigma^2}\Big)
\gg\frac{S_{w_2}\log x\,(\log_2x)^2}{y}\gg\log x\,\log_3x,
\]
where the last bound follows from \eqref{eq:By-bounds} and our assumption that
$S_{w_2}\ge\tfrac12W_y$. Lemma~\ref{lem:Bennett-ineq} now shows that
for some constant $c'>0$,
\[
\P\bigg(\bigg|\sum_{w_2<p\le w_3}X_p\bigg|\ge\frac{S_{w_2}}{2\log_3x}\bigg)
\le 2\exp\big\{-c'\log x\,\log_3x\big\}\ll x^{-100}.
\]
Thus, with probability at least $1-O(x^{-100})$ we have
\[
\sum_{w_2<p\le w_3}\big|\cS_{w_2}(y)\cap R_p\big|\le S_{w_2}
\bigg(\frac{1}{2\log_3x}+\sum_{w_2<p\le w_3}\frac1p\bigg)
\le\frac{S_{w_2}}{\log_3x}
\]
for sufficiently large $x$.
Recalling \eqref{eq:Sw2s3}, the proof is complete.
\end{proof}

\begin{lemma}[Sieving for $w_3 < p \leq w_4$]\label{lem:w3w4}
Let $w_3\defeq\log x\,(\log_2x)^2$ and
$w_4\defeq y^{4/3}$.
Conditional on $\cA_{w_3}$ satisfying
$S_{w_3}\ge\frac14W_y$, we have
\[
\P_{w_3,w_4}\bigg(\big|S_{w_4}-\tfrac38S_{w_3}\big|
\ge\frac{S_{w_3}}{(\log_2x)^{1/2}}\bigg)\ll x^{-100}.
\] 
\end{lemma}

\begin{proof}
Let $p_0\defeq w_3$ and let $p_1 < \ldots <p_m$
be the primes in $(w_3,w_4]$.
Using the notation \eqref{eq:Theta-notations}, we 
define random variables by
\[
X_j\defeq\Theta_{w_3,p_j}^{-1}S_{p_j}\qquad(j=0,1,\ldots,m).
\]
The sequence $X_0, X_1,\ldots,X_m$ is a martingale since
\[
\E(X_{j+1}|X_j)=\Theta_{w_3,p_{j+1}}^{-1}
\E(S_{p_{j+1}}|\cA_{p_j})
=\Theta_{w_3,p_{j+1}}^{-1}
\(1-p_{j+1}^{-1}\)S_{p_j}=X_j.
\]
Note that
\be\label{eq:Sw3-ests}
X_0=S_{w_3}\ge\tfrac14W_y\gg\frac{y\log_3x}{(\log_2x)^2},
\ee
where we have used \eqref{eq:By-bounds} in the last step.

We apply Azuma's inequality (Lemma~\ref{lem:Azuma-ineq}).
If $p_{j+1}>y$, then \change{$|X_{j+1}-X_j|\ll 1$ since $\Theta_{w_3,p_j}^{-1}\ll 1$}. In the case that
$p_{j+1} \le y$, Lemma~\ref{lem:sieve-upper} shows that
 for any value of $R_{p_{j+1}}$ we have
\begin{align*}
|X_{j+1}-X_j|&=\Theta_{w_3,p_j}^{-1}
\big|\(1-p_{j+1}^{-1}\)^{-1}S_{p_{j+1}}-S_{p_j} \big|
\ll \frac{S_{p_{j+1}}}{p_{j+1}}+S_{p_j} - S_{p_{j+1}} \\
&=\frac{S_{p_{j+1}}}{p_{j+1}}+\big|\cS_{p_j}(y) \cap R_{p_{j+1}}\big|
\ll \frac{y/p_{j+1}}{1+\log(y/p_{j+1})}.
\end{align*}
Consequently,
\[
\sum_{j=0}^{m-1}|X_{j+1}-X_j|^2 \ll
\frac{y^2}{w_3\log w_3 \log^2 y}+y^{4/3} \ll \frac{y^2}{\log x\, \log^5_2x}.
\]
Thus, if $c>0$ is sufficiently small,
then Lemma~\ref{lem:Azuma-ineq} shows that
\begin{equation}
\label{eq:Pw3w4}
\P_{w_3,w_4}\bigg(|X_m-X_0|\ge \frac{X_0}{(\log_2x)^{1/2}}\bigg)
\ll\exp\left\{-\frac{c\,X_0^2\log x\,(\log_2x)^4}{y^2}\right\}
\ll x^{-100}
\end{equation}
since by \eqref{eq:Sw3-ests} we have
\[
\frac{X_0^2\log x\,(\log_2x)^4}{y^2}\gg
\log x\,(\log_3x)^2.
\]
Using \eqref{eq:Theta-notations} and \eqref{eq:yx-relns} we write
\[
\lambda\defeq\Theta_{w_3,w_4}^{-1}=\tfrac83(1+r_x)\qquad
\text{with}\quad r_x\ll\frac{\log_3 x}{\log_2x};
\]
then noting that
\[
\big|S_{w_4}-\tfrac38S_{w_3}\big|=
\big|\lambda^{-1}X_m-\tfrac38X_0\big|=
\lambda^{-1}|X_m-(1+r_x)X_0|,
\]
for any $Z>0$ we have
\[
\P_{w_3,w_4}\big(\big|S_{w_4}-\tfrac38S_{w_3}\big|\ge Z\big)
\le\P_{w_3,w_4}\big(\big|X_m-X_0\big|\ge \lambda Z-r_xX_0\big).
\]
In view of \eqref{eq:Pw3w4} this implies that
\[
\P_{w_3,w_4}\big(\big|S_{w_4}-\tfrac38S_{w_3}\big|\ge Z\big)
\ll x^{-100}
\]
holds provided that
\[
\lambda Z-r_xX_0\ge\frac{X_0}{(\log_2x)^{1/2}}.
\]
The result follows by taking $Z\defeq\frac{X_0}{(\log_2 x)^{1/2}}
=\frac{S_{w_3}}{(\log_2 x)^{1/2}}$ and noting that $\lambda\ge 2$.
\end{proof}


{\Large\section{Random sieving by large primes}\label{sec:random-large-primes}}

In this section, we adopt the notation
\[
S_w\defeq|\cS_w(y)|=|[0,y] \cap \cS_w|
\]
from the previous section; however, we \emph{do not} assume inequalities \eqref{eq:xy} and \eqref{eq:yx-relns},
except in Corollary \ref{cor:w1w5} below.
We do assume that $y$ is sufficiently large.
Sieving by large primes ($p>y^4$, say) is easier because there is a
relatively low probability that $\cS\cap R_p\ne\varnothing$
and we are able to deploy combinatorial methods.

\medskip

\begin{lemma}[Sieving for $w_4 < p \leq w_5$]\label{lem:w4w5}
Let $v$ be a real number greater than
$w_4\defeq y^{4/3}$, and let 
$\vartheta\in[y^{-1/4},1)$.
Conditional on~$\cA_{w_4}$, we have 
\[
\P_{w_4,v} \Big(\big|S_v-\Theta_{w_4,v}S_{w_4}\big| 
\ge \vartheta S_{w_4} \Big) \le \exp\{-0.1\vartheta^2 S_{w_4}\}.
\]
\end{lemma}

\begin{proof}
Put $\cS\defeq\cS_{w_4}(y)$, $\ell\defeq |\cS|=S_{w_4}$,
and let $\ccP$ be the set of primes in $(w_4,v]$.
The random residue classes $\{R_p:p\in\ccP\}$
give rise to a bipartite graph $\ccG$ that has vertex sets $\cS$ and $\ccP$,
with edges connecting the vertices $s\in\cS$ and $p\in\ccP$ if and
only if $s\in R_p$ (i.e., $s\equiv a_p\bmod p$).
\change{Since $0 \le s\le y < w_4$, for every $p$ there is at most one vertex $s$
joined to it.}
 For any $s\in\cS$, let
$d(s)$ be its degree,
\[
d(s)\defeq\big|\{p\in\ccP:s\in R_p\}\big|,
\]
and let $\cS^+$ be the set of vertices in $\cS$ of positive degree:
\[
\cS^+\defeq\{s\in\cS:d(s)>0\}
=\bigcup_{p\in\ccP}(\cS\cap R_p).
\]
Finally, we denote by $\dd$ the vector $\big\langle d(s):s\in\cS^+\big\rangle$.
In this manner, the random residue classes $\{R_p:p\in\ccP\}$ determine
a subset $\cS^+\subset\cS$ and a vector $\dd$.

For any subset $\cT=\{t_1,\ldots,t_m\}$ in $\cS$
and a vector $\rr=\langle r_1,\ldots,r_m\rangle$
whose entries are positive integers,
let $E(\cT,\rr)$ be the event that the random graph $\ccG$ described above
has $\cS^+=\cT$ and $\dd=\rr$.
 Since $\cS\subset[0,y]$ and $w_4>y$, we have $|\cS\cap R_p|\le 1$
for all $p\in\ccP$, and thus
\[
h\defeq r_1+\cdots+r_m=\sum_{s\in\cS^+}d(s)
=\big|\{p\in\ccP:\cS\cap R_p\ne\varnothing\}\big|.
\]
Fixing the primes $p_1,\ldots,p_h\in \ccP$ with $R_p \cap \cS\ne \varnothing$, there are $\binom{h}{r_1\, \cdots\, r_m}$ ways to
choose the graph's edges connecting the $p_i$ to $\cT$.
Consequently,
\begin{align}
\nonumber
&\P_{w_4,v}(E(\cT,\rr))
=\ssum{p_1,\ldots,p_h\in\ccP\\p_1<\cdots<p_h}\frac{1}{p_1\cdots p_h}
\binom{h}{r_1\;r_2\;\cdots\;r_m}
\prod_{p\in\ccP\setminus\{p_1,\ldots,p_h\}}\bigg(1-\frac{\ell}{p}\bigg)\\
\label{eq:Er}
&\qquad\qquad=\binom{h}{r_1\;r_2\;\cdots\;r_m}
\prod_{p\in\ccP}\bigg(1-\frac{\ell}{p}\bigg)
\ssum{p_1,\ldots,p_h\in\ccP\\p_1<\cdots<p_h}
\prod_{j=1}^h\frac{1}{p_j-\ell}.
\end{align}
Relaxing the conditions on the last sum in \eqref{eq:Er}, we find that
\[
\P_{w_4,v}(E(\cT,\rr)) \le \frac{V U^h}{r_1!\cdots r_m!}\qquad\text{with}\quad
V\defeq\prod_{p\in\ccP}\Big( 1-\frac{\ell}{p}\Big)\quad\text{and}\quad
U\defeq\sum_{p\in\ccP}\frac{1}{p-\ell}.
\]
For fixed $m$, there are $\binom{\ell}{m}$ choices for $\cT$;
thus, summing over all $r_1,\ldots,r_m$ we conclude that
\be\label{eq:PWU}
\P_{w_4,v}(S_{w_4}-S_v=m) \le V\binom{\ell}{m}(e^U-1)^m.
\ee

The complete sum over $m$ of the right side of \eqref{eq:PWU} is equal to \change{$V e^{U\ell}$}, and the peak occurs when
$m=(1-e^{-U})\ell+O(1)$.
We also have
\be\label{eq:eU}
1-e^{-U}=1-\Theta_{w_4,v}\bigg(
1+O\bigg(\frac{\ell}{w_4\log w_4}\bigg)\bigg),
\ee
Standard large-deviation results for the binomial distribution
(such as Lemma~\ref{lem:Azuma-ineq}) imply that \change{for any $\delta>0$,}
\[
e^{-U\ell} \sum_{|m-(1-e^{-U}) \ell| \ge \delta \ell}
\binom{\ell}{m} (e^U-1)^m \le 2 e^{-\delta^2 \ell/2}.
\]
Recalling that $\ell\defeq S_{w_4}$, we see that the inequality
\[
\big|S_v-\Theta_{w_4,v}\ell|\ge\vartheta\ell 
\]
implies via \eqref{eq:eU} that
\begin{align*}
|m-\change{(1-e^{-U})} \ell|
&\ge \vartheta \ell - |e^{-U}-\,\Theta_{w_4,v}|\ell 
\ge \vartheta \ell - O(y^{-1/3}\ell)
 \ge \vartheta \ell /2
\end{align*}
for all large $x$ since $\change{w_4}\defeq y^{4/3}$ and $\ell\le y$.
Combining our results above, we conclude that
\dalign{
\P_{w_4,v}\(\big|S_v-\Theta_{w_4,v}\ell\big|\ge\vartheta\ell\)
&\ll \change{V e^{U\ell} e^{-\vartheta^2 \ell/8}} \\
&\ll \change{e^{-\vartheta^2\ell/8+O(\ell^2/w_4)}} \\
&\le e^{-\vartheta^2\ell/10}
}
for all large $x$, and the proof is complete.
\end{proof}

Combining Lemmas~\ref{lem:w1w2}, \ref{lem:w2w3}, \ref{lem:w3w4} and
\ref{lem:w4w5} (with $v\defeq y^8$ and
$\vartheta\defeq y^{-1/10}$) we obtain the following result.

\begin{corollary}[Sieving for $w_1 < p \leq w_5$]\label{cor:w1w5}
Assume \eqref{eq:xy},
let $w_1\defeq(y/\log y)^{1/2}$ and $w_5\defeq y^8$.
Conditional on $\cA_{w_1}$, we have
with probability $1-O(x^{-100})$ that
\[
\left|S_{w_5}-\frac{S_{w_1}}{16}\right|\ll_{\alpha,\beta}
\frac{\log_4 x}{\log_3 x}S_{w_1}.
\]
\end{corollary}

Our next result is a very general tool for handling primes
larger than $y^4$.

\begin{lemma}[Sieving for $w_5 < p \leq z$, I]\label{lem:SzSP}
Let \red{$y^4 \le w < z$, $y\ge (\log x)^{1/2}$} and let $\ccP$ be a set of primes \red{in $(w,z]$} 
such that $\sum_{p\in \ccP} 1/p \ge 1/10$. Let $\cS\subseteq \cS_w$
with $|\cS| \le 10y$, and such that for all
$p\in \ccP$, $\cS$ is distinct modulo $p$.
Conditional on $\cA_w$, we have
for all $0\le g\le |\cS|$:
\[
\P_{\ccP}\(\Big|\cS \setminus \bigcup_{p\in \ccP} R_p\Big|=g\)=
(1-\Theta)^{|\cS|-g}\Theta^g \binom{|\cS|}{g}(1+\red{O(y^3/w)}),
\]
where
\[
\Theta \defeq \prod_{p\in \ccP} (1-1/p).
\]
\end{lemma}

\begin{proof}
Put $\ell\defeq|\cS|$,
and assume that $\ell\ge 1$
(the case $\ell\defeq 0$ being trivial). Take $m\defeq\ell-g$, and let
$\cT$, $\rr$, $E(\cT,\rr)$ and $h$ be defined as
in Lemma~\ref{lem:w4w5} with $|\cT|=m=\ell-g$. 
As before (see \eqref{eq:Er}) we have
\be\label{eq:Er2}
\P_\ccP(E(\cT,\rr))
=\binom{h}{r_1\;r_2\;\cdots\;r_m}
\prod_{p\in\ccP}\bigg(1-\frac{\ell}{p}\bigg)
\ssum{p_1,\ldots,p_h\in\ccP\\p_1<\cdots<p_h}
\prod_{j=1}^h\frac{1}{p_j-\ell}.
\ee
\red{
Let $T_h$ be the sum over $p_1,\ldots,p_h$
in \eqref{eq:Er2}. 
Summing over all vectors $\rr$, we find that
\begin{align*}
\change{\P_\ccP\big(\big| \cS \setminus \cup_{p\in \ccP} R_p \big|=\ell-m \big)}&=
\ssum{\cT\subset\cS\\|\cT|=m} \sum_h \sum_{r_1+\cdots+r_m=h} \binom{h}{r_1 \, \cdots \, r_m} V T_h\\
&=V \binom{\ell}{m} 
\ssum{r_1,\ldots,r_m\ge 1 \\ h:=r_1+\cdots+r_m} \frac{h! T_h}{r_1!\cdots r_m!},
\end{align*}
where
\[
V\defeq\prod_{p\in\ccP}\Big(1-\frac{\ell}{p}\Big).
\]
When $m=0$, the sum on the right side is interpreted to be 1.
We have
\dalign{
T_h&=\frac{1}{h!}\bigg(\sum_{p\in\ccP}\frac{1}{p-\ell}\change{+O\pfrac{h}{w}}\bigg)^h\\
&=\frac{1}{h!}\bigg(\sum_{p\in\ccP}\frac{1}{p}
+O\bigg(\frac{\red{h+\ell}}{\change{w}}\bigg)\bigg)^h\\
&=\frac{(-\log \Theta+\red{O(y^2/w)})^h}{h!},
}
provided that $h\le y^2$.  For any $h$ we also have the crude upper bound
\[
T_h \le \frac{1}{h!}\bigg(\sum_{p\in\ccP}\frac{1}{p-\ell}\bigg)^h \le \frac{(\log_2 x)^h}{h!}.
\]

Assuming that $m\ge 1$, let
\[
\lambda = \frac{y^2}{m \log_2 x}.
\]
As $m\le 10y$, we have $\lambda \ge \frac{y}{10\log_2 x}\ge \frac{(\log x)^{1/2}}{10\log_2 x}$.
Thus,
\begin{align*}
\ssum{r_1,\ldots,r_m\ge 1 \\ h:=r_1+\cdots+r_m>y^2} \frac{h! T_h}{r_1!\cdots r_m!} &\le
\sum_{r_1,\ldots,r_m\ge 0} \frac{(\log_2 x)^{r_1+\cdots+r_m}}{r_1!\cdots r_m!} \lambda^{r_1+\cdots+r_m-y^2}\\
&=e^{m\lambda \log_2 x-y^2 \log \lambda}=e^{y^2-y^2 \log \lambda} < e^{- 2y^2}
\end{align*}
if $x$ is large enough.  It follows that
\begin{align*}
\ssum{r_1,\ldots,r_m\ge 1 \\ h:=r_1+\cdots+r_m} \frac{h! T_h}{r_1!\cdots r_m!} &= O(e^{-2y^2})+
\ssum{r_1,\ldots,r_m\ge 1} \frac{(-\log\Theta+O(y^2/w))^{r_1+\cdots+r_m}}{r_1!\cdots r_m!} \\
&= O(e^{-2y^2})+\Big( e^{-\log\Theta+O(y^2/w)} - 1 \Big)^m\\
&= O(e^{-2y^2})+\big( 1 + O(y^3/w) \big) \big( \Theta^{-1}-1 \big)^m\\
&=\big( 1 + O(y^3/w) \big) \big( \Theta^{-1}-1 \big)^m,
\end{align*}
using in the last step that $(\Theta^{-1}-1)^m\ge 10^{-10 y}$ and $w\le x \le e^{y^2}$.
Finally,
\[
V\defeq\prod_{p\in\ccP}\Big(1-\frac{\ell}{p}\Big)=\Theta^\ell(1+O(y^2/w))
\]
and this completes the proof.
}  
\end{proof}

\begin{corollary}[Sieving for $w_5 < p \leq z$, II]\label{cor:binomSk}
\red{Let $y\ge (\log x)^{1/2}$.}
Uniformly for $z^{1/2} \ge w \ge y^4$, we have
\[
\E_{w,z} \binom{S_z}{k}=\Theta_{w,z}^k\binom{S_w}{k}(1+\red{O(y^3/w)}).
\]
\end{corollary}

\begin{proof}
Let $\Theta\defeq\Theta_{w,z}$.
By Lemma~\ref{lem:SzSP} with $\cS\defeq\cS_w\cap [0,y]$ and $\ccP$ the set of primes in $(w,z]$, we have
\begin{align*}
\E_{w,z}\binom{S_z}{k}&=(1+\red{O(y^3/w)})\sum_{g=k}^{S_w}
(1-\Theta)^{S_w-g}\Theta^g \binom{S_w}{g}\binom{g}{k}\\
&=(1+\red{O(y^3/w)})\Theta^k\binom{S_w}{k}\sum_{j=0}^{S_w-k}
(1-\Theta)^{S_w-k-j}\Theta^j \binom{S_w-k}{S_w-k-j}\\
&=(1+\red{O(y^3/w)})\Theta^k\binom{S_w}{k}.\qedhere
\end{align*}
\end{proof}

The next lemma has a weaker
conclusion than Lemma~\ref{lem:SzSP}
but is more general and is needed for a second moment argument below
in which we derive a lower bound for the largest prime gap in $[0,x]$.

\begin{lemma}[Sieving for $w_5 < p \leq z$, III]\label{lem:VarSP}
Let $w$ and $z$ be real numbers for which
$z^{1/2}\ge w\ge y^8$.
Let \change{$\cS \subset \cS_w \cap [0,e^y]$} with $|\cS|\le y$
and such that for every prime $p>w$, no more than two
numbers in $\cS$ lie in any given residue class modulo $p$.
Then
\[
\P_{w,z}\bigg(\cS \cap \cS_z=\varnothing\bigg)=
(1-\Theta_{w,z})^{|\cS|}(1+O(y^4/w)).
\]
\end{lemma}

\begin{proof}
Put $\ell\defeq|\cS|$, and let $\ccP$ be the set of primes in $(w,z]$, and put
\[
\cQ\defeq\big\{p\in\ccP:p\mid s-s'\text{~for some~}s,s'\in \cS, s\ne s'\big\}.
\]
Note that the bound
\be\label{eq:|P|bd}
\change{|\cQ|\le\frac{\ell^2 y}{\log w} \le y^{3}}
\ee
holds if \change{$y$} is large enough.

By assumption, for every $p\in \cQ$, \change{$|\cS \cap R_p| \le 2$.}
Let $E_m$ be the event that for $\cS\cap R_p\ne\varnothing$
holds for precisely $m$ primes $p\in\cQ$.
Since for any prime $p\in\ccP$ the probability that
$\cS\cap R_p\ne\varnothing$ does not exceed $\ell/p$,
using \eqref{eq:|P|bd} we have
\be\label{eq:Am}
\P_{\cQ}(E_m)\le\frac{1}{m!}
\bigg(\sum_{p\in\cQ}\frac{\ell}{p}\bigg)^m
\le\pfrac{e\ell|\cQ|}{mw}^m\le \change{(ey^4/w)^m}\qquad(m\ge 1).
\ee
Assume the event $E_m$ occurs, and fix $\cA_\cQ$.
If $\cS$ has precisely $n$ elements covered by $\bigcup_{p\in\cQ}R_p$,
then $0\le n\le 2m$, the upper bound
being a consequence of our hypothesis on $\cS$. Put
\[
\cS'\defeq\big\{s\in\cS:s\not\in R_p\text{~for all~}p\in\cQ\big\},
\]
so that $|\cS'|=\ell-n$.
Lemma~\ref{lem:SzSP} implies that
\begin{align*}
\P_{\ccP\setminus\cQ}
\bigg(\cS'\subset\bigcup_{p\in\ccP\setminus\cQ}R_p\bigg)
&=(1+\red{O(y^3/w)})\Big(1-\Theta_{w,z}
\prod_{p\in\cQ}\big(1-p^{-1}\big)^{-1}\Big)^{\ell-n}\\
&=(1+O(y^4/w))\Big(1-\Theta_{w,z}\Big)^{\ell-n}\\
&\ll \Big(1-\Theta_{w,z}\Big)^{\ell-2m},
\end{align*}
since 
\change{
\[
\prod_{p\in \cQ}\big(1-p^{-1}\big)^{-1}=1+O(|\cQ|/w)=1+O(y^3/w)
\]
by \eqref{eq:|P|bd}.}
Now $\P_\cQ(E_0)=1-O(y^4/w)$ by \eqref{eq:Am}, so we conclude that
\begin{align*}
&\P_{w,z}\bigg(\cS\subset\bigcup_{p\in\ccP}R_p\bigg)
=\sum_{m=0}^{|\cQ|}\P_{\cQ}(E_m)
\cdot\E_\cQ\bigg(\P_{\ccP \setminus \cQ}\bigg(\cS'\subset
\bigcup_{p\in\ccP \setminus \cQ}R_p\bigg)\Big|E_m\bigg)\\
&\qquad=(1+O(y^4/w))\big(1-\Theta_{w,z}\big)^\ell
+O\bigg(\sum_{m\ge 1}\change{(ey^4/w)^m}
\big(1-\Theta_{w,z}\big)^{\ell-2m}\bigg)\\
&\qquad=(1+O(y^4/w))\big(1-\Theta_{w,z}\big)^\ell.
\end{align*}
This completes the proof.
\end{proof}


{\Large\section{The behavior of the largest gap}
\label{sec:large-gaps-from-model}}

In this section we use the estimates from the previous section
to complete the proof of Theorem~\ref{thm:conditional}.
\change{In Theorems \ref{thm:gap-up} and \ref{thm:gap-down} below}, we suppose that 
\be\label{epsx}
\eps=\eps(x)\defeq \change{\frac{1}{(\log_3 x)^{1/3}}}.
\ee
 We also note that
 \[
u < W_{g(u)+1}\log(g(u)+1) \le (W_{g(u)}+1)\log(g(u)+1).
\]
 and hence
\be\label{Wgu}
W_{g(u)} \log g(u)=u+O(\log u).
\ee

\medskip

\begin{theorem}[Probabilistic upper bound for gap]\label{thm:gap-up}
For large $x$,
\[
\P\big[G_\cR(x)\le g\big((1+\eps)\xi (\log \tfrac{x}{2})^2\big)\big]\ge 1-x^{-\eps/2}. 
\]
\end{theorem}

\begin{theorem}[Probabilistic lower bound for gap]\label{thm:gap-down}
If $x$ is large then
\[
\P\big[G_\cR(x)\ge g\big((1-\eps)\xi (\log 2x)^2\big)\big]\ge 1-O\big((\log x)^{-8}\big). 
\]
\end{theorem}

\begin{proof}[Proof of Theorem~\ref{thm:gap-up}]
Let $y\defeq g((1+\eps)\xi (\log \tfrac{x}{2})^2)$, so that by
\eqref{Wgu} we have
\be\label{Wyx-upp}
W_y \log y=(1+\eps)\xi (\log x)^2+O(\log x).
\ee
We also have by \eqref{eq:By-bounds} the bounds
\[
\log^2 x \ll y \ll (\log^2 x)\log_2 x.
\]
Let $z\defeq z(x)$. The probability that $\cR\cap [0,x]$ 
has a gap of size $\ge y$ does not exceed the probability
that $\cS_z\cap [0,x]$ has a gap of size $\ge y$, which in turn
is at most 
\change{
\[
\E\big|\{n\le x:[n,n+y]\cap\cS_z=\varnothing\}\big|\le x\cdot\PR(\cS_z=0).
\]
}

Let $w_1\defeq(y/\log y)^{1/2}$ and $w_5\defeq y^8$ as before.
Also put $\eta \defeq \frac{\log_4 x}{\log_3 x}$.
Applying Corollary~\ref{cor:w1w5} together with \eqref{Wyx-upp}, it follows that 
with probability $1-O(x^{-100})$ we have
\dalign{
S_{w_5}&=(1+O(\eta))\frac{S_{w_1}}{16}
\ge(1+O(\eta))\frac{W_y}{16} \\
&\ge \frac{(1+\eps+O(\eta))\,\xi (\log x)^2}{32\log_2x}\\
&\ge \frac{(1+2\eps/3)\,\xi (\log x)^2}{32\log_2x}
}
using \eqref{epsx} in the final step.
Fix $\cA_{w_5}$ so that $S_{w_5}$ satisfies this inequality.
Taking into account that
\[
\Theta_{w_5,z}=\frac{32\log_2x}{\xi\log x}\(1+O\pfrac{1}{\log_2 x}\),
\]
Lemma~\ref{lem:SzSP} now shows that
\[
\P_{w_5,z}(S_z=0)\ll(1-\Theta_{w_5,z})^{\change{S_{w_5}}}\ll x^{-1-\eps/2},
\]
as required.
\end{proof}

\begin{proof}[Proof of Theorem~\ref{thm:gap-down}]
Set $y\defeq g((1-\eps)\xi (\log 2x)^2)$, so that
\be\label{Wxy-lwr}
W_y \log y=(1-\eps)\xi \log^2 x+O(\log x).
\ee
Again, \eqref{eq:By-bounds} implies that
\[
\log^2 x \ll y \ll (\log^2 x) \frac{\log_2 x}{\log_3 x}.
\]
Let $z\defeq z(x/2)$, 
$w_1\defeq(y/\log y)^{1/2}$,
$w_5\defeq y^8$ and $\eta\defeq \frac{\log_4 x}{\log_3 x}$. In particular, $z\sim (x/2)^{1/e^{\gamma}}$ by 
\eqref{eq:theta}, \change{and
\be\label{eq:w1-upper}
w_1 \ll \frac{\log x}{(\log_3 x)^{1/2}}.
\ee
}
It suffices to show that with high probability,
$\cS_z\cap (x/2,x]$ has a gap of size $\ge y$, for this
implies that $\cR$ has a gap of size $\ge y$ within $[0,x]$.
For the sake of brevity we write
\[
\cF(u,v) \defeq [u,u+y] \setminus \bigcup_{p\le v} R_p,
\qquad F(u,v)\defeq|\cF(u,v)|.
\]
That is, $F(u,v)$ counts the number of elements in $[u,u+y]$
sieved by the primes $\le v$. In particular, $S_w=F(0,w)$.
\change{There is some vector $(b_p)_{p\in w_1}$ so that
there are exactly $W_y$ integers in $[0,y]$ that avoid the
residue classes $(b_p \bmod p)_{p\le w_1}$.}
Setting
\[
Q\defeq\prod_{p\le w_1}p,
\]
\change{for any $\cA_{w_1}$,} there is a progression $b\bmod Q$ such that 
\[
F(u,w_1)=W_y \qquad\text{whenever}\quad u\equiv b\bmod Q.
\]
Specifically, choose $b$ such that
\change{$b\equiv a_p-b_p\bmod p$} for all primes $p\le w_1$.
Let $\cU$ be the set of integers $u\equiv b\bmod Q$ such that
$[u,u+y]\subset (x/2,x]$. 
 We show that with high probability,
$F(u,z)=0$ for at least one $u\in\cU$.

By Corollary~\ref{cor:w1w5}, with probability at least $1-O(x^{-100})$,
we have for any given $u\in\cU$ the bound
\be\label{eq:Fw1w5}
F(u,w_5)=(\tfrac{1}{16}+O(\eta)) F(u,w_1)=(\tfrac{1}{16}+O(\eta)) W_y.
\ee
Let $E$ be the event that this bound holds for
\emph{every} $u\in \cU$.
By the union bound, $\P_{w_1,w_5}(E)\ge 1-O(x^{-99})$.
Conditioning on $E$, we denote
\[
\cU_r\defeq\{u\in\cU:F(u,w_5)=r \}\qquad(r\ge 0).
\]
The sets $\cU_r$ depend only on $\cA_{w_5}$, and $\cU_r=\varnothing$
unless $r=(\tfrac{1}{16}+O(\eta))W_y$ by \eqref{eq:Fw1w5}. 
Rather than work with all $r$, we focus on a popular value of $r$; thus,
let $\ell$ be fixed with the property that $|\cU_\ell|\ge|\cU_r|$
for all $r$. \change{By \eqref{eq:w1-upper}, we have}
\be\label{eq:cardUl}
|\cU_\ell|\gg \frac{|\cU|}{\eta W_y} \gg \frac{x}{QW_y}
=x e^{-O(w_1)}\gg x^{1-O((\log_3 x)^{-1/2})}.
\ee
Combining \eqref{Wxy-lwr} with \eqref{eq:Fw1w5} and \eqref{epsx}, we have
\be\label{eq:l-low2}
\ell\le(\tfrac{1}{16}+O(\eta))W_y
\le \frac{(1-(2/3)\eps)\xi(\log x)^2}{32\log_2x}.
\ee

Next, let
\[
M\defeq\big|\{u\in\cU_\ell:F(u,z)=0\}\big|,
\]
which counts those intervals indexed by \change{$u\in\cU_\ell$ for which
$\cF(u,w_5)$ is} 
covered by $\bigcup_{w_5<p\le z} R_p$.
We analyze $M$ using first and second moments.
Firstly, by Lemma~\ref{lem:SzSP},
\[
\E_{w_5,z}M=\sum_{u\in\cU_\ell}\P_{w_5,z}(F(u,z)=0)
=|\cU_\ell|(1-\Theta)^\ell(1+\red{O(y^3/w_5)}),
\]
where
\be\label{eq:Theta-boom!}
\Theta\defeq\Theta_{w_5,z}=\frac{32\log_2x}{\xi\log x}
\(1+O\pfrac{1}{\log_2 x}\).
\ee
To bound the second moment of $M$, apply Lemma~\ref{lem:VarSP} with
$\cS\defeq\cF(u,w_5) \cup \cF(u',w_5)$, where 
$u$ and $u'$ are distinct elements of $\cU_\ell$.
The hypotheses of Lemma~\ref{lem:VarSP} are satisfied
as any prime $p>w_5>y$ can divide at most two elements of
$\cS$.
We obtain
\begin{align*}
\E_{w_5,z} M^2 &=\E_{w_5,z} M+\ssum{u,u'\in \cU_\ell \\ u \ne u'} \P_{w_5,z} \( F(u,z)=F(u',z)=0 \) \\
&=|\cU_\ell|^2(1-\Theta)^{2\ell} (1+O(y^4/w_5))+O\big(|\cU_\ell|(1-\Theta)^\ell\big).
\end{align*}
By \eqref{eq:cardUl}, \eqref{eq:l-low2} and \eqref{eq:Theta-boom!}
we have
\[
|\cU_\ell|(1-\Theta)^\ell\ge x^{2\eps/3 - O((\log_3 x)^{-1/2})}
\ge x^{\eps/2}
\]
for large $x$, 
and hence we bound the variance by
\[
\sigma^2\defeq \V_{w_5,z}M=
\E_{w_5,z} M^2-(\E_{w_5,z} M)^2
\ll|\cU_\ell|^2(1-\Theta)^{2\ell}y^4/w_5.
\]
Thus, Chebyshev's inequality implies 
\[
\P_{w_5,z} \(M \ge \tfrac12 |\cU_\ell| (1-\Theta)^\ell \)
\ge 1-O(y^4/w_5)=1-O(1/y^4).
\]
In particular, with probability at least
$1-O(y^{-4})=1-O((\log x)^{-8})$ there is an interval $[u,u+y]$ in $(x/2,x]$
completely sieved out by $\cA_z$.
\end{proof}

\begin{proof}[Proof of Theorem~\ref{thm:conditional}]
Let $x_j\defeq 2^j$ vary over positive integers $j$,
and let $\eps>0$ be fixed. 
Theorem~\ref{thm:gap-up} implies that for large $j$ we have
\[
\P \big[ G_\cR(x_j) \le g((1+\eps)\xi \log^2 x_{j-1}) \big] \ge 1- x_j^{-\eps/2} \qquad (j \text{ large}).
\]
The convergence of $\sum_j x_j^{-\eps/2}$ implies, via the 
Borel-Cantelli lemma, that almost surely there is a $J$ so that
\[
 G_\cR(x_j) \le g((1+\eps)\xi \log^2 x_{j-1}) \qquad (j\ge J).
\]
As $G_\cR$ and $g$ are both increasing functions, the above relation
implies that for all $x_{j-1} < x \le x_j$ and $j>J$ we have
\[
G_\cR(x) \le G_\cR(x_j) \le g((1+\eps)\xi \log^2 x_{j-1})
\le g((1+\eps)\xi \log^2 x),
\]
In a similar manner, Theorem~\ref{thm:gap-down} and Borel-Cantelli
imply that almost surely there is a $J$ so that
\[
 G_\cR(x_j) \ge g((1-\eps)\xi \log^2 x_{j+1}) \qquad (j\ge J).
\]
As before, this implies that
\[
G_\cR(x) \ge g\big( (1-\eps)\xi \log^2 x\big) \qquad (x\ge x_{J}).
\]
\end{proof}


{\Large\section{Large gaps from Hardy-Littlewood}\label{sec:HLA}}

To prove Theorems~\ref{thm:HLgaps} and \ref{thm:HLgaps-avg},
we start with a simple inclusion-exclusion result (a special case of the Bonferroni inequalities or the ``Brun pure sieve'').

\bigskip

\begin{lemma}[Brun's sieve]\label{lem:UK}
Suppose that $y \geq 1$, let $\cN, \cA$ be sets of positive integers,
and put
\[
T\defeq\sum_{n\in\cN}\prod_{h\in[0,y]}(1-\ind{\cA}(n+h))
\]
and
\[
U_K\defeq \sum_{k=0}^K(-1)^k\ssum{\cH\subset[0,y]\\|\cH|=k} \;
\sum_{n\in\cN} \;\prod_{h\in\cH} \ind{\cA}(n+h)\qquad
(K\ge 0).
\]
Then, for any even $K$ we have $T\le U_K$,
and for any odd $K$ we have $T\ge U_K$.
\end{lemma}

\begin{proof}
For any integers $K,m\ge 0$ let
\[
\delta_K(m)\defeq\sum_{k=0}^K(-1)^k\binom{m}{k}\mand
\delta(m)\defeq\begin{cases}
1&\quad\hbox{if $m=0$},\\
0&\quad\hbox{if $m\ge 1$}.
\end{cases}
\]
Observe that
\[
\delta(m)\le\delta_K(m)\quad\text{($K$ even)}\mand
\delta(m)\ge\delta_K(m)\quad\text{($K$ odd)};
\]
hence, taking $A(n)\defeq\big|\{0\le h\le y:n+h\in\cA\}\big|$ we have
\[
T=\sum_{n\in\cN}\delta(A(n))=\sum_{n\in\cN} \delta_K(A(n))+\theta,
\]
where $\theta\ge 0$ if $K$ is even and $\theta\le 0$ if $K$ is odd. Also,
\[
\sum_{n\in\cN} \delta_K(A(n))
=\sum_{k=0}^K(-1)^k\sum_{n\in\cN}\binom{A(n)}{k}=U_K
\]
since
\[
\binom{A(n)}{k}=\sum_{\substack{\cH\subset[0,y] \\|\cH|=k}}
\;\prod_{h\in\cH}\ind{\cA}(n+h)\qquad(n\in\cN),
\]
and the lemma is proved.
\end{proof}

\medskip

\begin{proof}[Proof of Theorem~\ref{thm:HLgaps}]
Although Theorem~\ref{thm:HLgaps} concerns the behavior
of a specific set $\cA$, our first task is to express the
gap-counting function for $\cA$ in terms of the random quantities
with which we have been working in the past few sections.

First, observe that \eqref{eq:HLA} with $\cH=\{0\}$ implies that
\[
\big|\{n\le x:n\in\cA\}\big|\sim x/\log x,
\]
and it follows trivially that $G_\cA(x) \gg \log x$. Therefore,
by adjusting the implied constant in the conclusion of the theorem, we may assume that 
\be\label{cD}
\kappa \ge D \frac{\log_2 x}{\log x}
\ee
 for a sufficiently large constant $D$.

Let $x$ be a large real number, put $\cN\defeq[x/2,x]$ 
and let $y,K$ be integer parameters to be chosen later, with $K$ odd and with $K\le \frac{\kappa \log x}{2\log_2 x}$.
Define $T$ and $U_K$ as in Lemma~\ref{lem:UK}.
Since $T\ge U_K$ by Lemma~\ref{lem:UK},
our aim is to show that $U_K\ge 1$. Using \eqref{eq:HLA} we see that
\[
U_K=\sum_{k=0}^K(-1)^k\int_{x/2}^{x}\frac{1}{(\log t)^k}
\ssum{\cH\subset[0,y]\\|\cH|=k}\fS(\cH)\,dt+O(E),
\]
where
\[
E\defeq Kx^{1-\kappa}\binom{y+1}{K}.
\]
\change{By Lemma~\ref{lem:VH}}, replacing $\fS(\cH)/\log^k t$ with $\TrunSS{z(t)}$ induces an additive error of size $O(E)$
since $\kappa\le 1/2$.
Also, \eqref{VHzSz} implies that
\[
\ssum{\cH\subset[0,y]\\|\cH|=k}\TrunSS{z(t)}
=\E_{z(t)}\binom{S_{z(t)}}{k},
\]
and we get
\[
U_K=\int_{x/2}^{x}\E_{z(t)}\sum_{k=0}^K(-1)^k\binom{S_{z(t)}}{k}\, dt+O(E).
\]
Since $K$ is odd, the sum on $k$ is a lower bound for
$\P(S_{z(t)}=0)$; adding the term $k=K+1$ switches the
inequality (cf. the proof of Lemma~\ref{lem:UK}) and thus
\be\label{UK1}
U_K \ge \int_{x/2}^{x} \P(S_{z(t)}=0) - 
\E_{z(t)}\binom{S_{z(t)}}{K+1}\, dt+O(E).
\ee

Let
\[
w\defeq y^4,\qquad
z\defeq z(x/2).
\]

The upper bound sieve (Lemma~\ref{lem:sieve-upper}) implies
the crude bound $S_w \le Cy/\log y$ for some absolute constant $C$.
We now put
\be\label{eq:HLy_and_K}
y\defeq\frac{\kappa \,\xi \log^2 x}{400 C \log_2x}
\mand
K\defeq 2\fl{\frac{100 C y}{\log x}}-1.
\ee
With these choices, $K \le \frac{\kappa\log x}{2\log_2 x}$
and, using \eqref{cD}, we have 
\be\label{yD}
y \ge \frac{D}{400C}\log x.
\ee
It also follows that
\[
E \ll x^{1-\kappa} (\log x)^{K} \ll x^{1-\kappa+\kappa\,\xi/2}\ll x^{1-\kappa/3}.
\]
for all large $x$.
Corollary \ref{cor:binomSk}
and the crude bound $\Theta_{w,z} \le 8\frac{\log y}{\log x}$
imply that
\dalign{
\E_{z(t)}\binom{S_{z(t)}}{K+1} &\le \E_{z}\binom{S_{z}}{K+1}\\
&\ll \Theta_{w,z}^{K+1} \E_w \binom{S_w}{K+1}\\
&\ll \(\Theta_{w,z} \frac{eCy}{K\log y}\)^{K+1} \\
&\ll e^{-K} \ll e^{-200Cy/\log x},
}
where we used \eqref{eq:HLy_and_K} in the last step.
It remains to show that $\P_{z(t)}(S_{z(t)}=0)$ is substantially larger.
 Lemma~\ref{lem:SzSP} implies immediately that
\dalign{
\P_z(S_{z(t)}=0) &\ge \P_{z} (S_z=0) \gg (1-\Theta_{w,z})^{S_w} \\
&\gg e^{-\Theta_{w,z} (Cy/\log y)} \ge e^{-8C y/\log x},
}
as required. Combining these estimates with \eqref{UK1} gives
\[
U_K \gg x e^{-8C y/\log x}+O(xe^{-200Cy/\log x}+x^{1-c/3}) \gg x e^{-8C y/\log x},
\]
the last inequality following from \eqref{yD}, the fact that
$D$ is sufficiently large, and that $y/\log x \ll \kappa/\log_2 x$.
This completes the proof of 
 Theorem~\ref{thm:HLgaps}.
\end{proof}


\begin{proof}[Proof of Theorem~\ref{thm:HLgaps-avg}]
Let $x$ be large, let $\eps>0$, and let
$y\defeq g((1-\eps)c\,\xi\log^2 x).$ By \eqref{Wgu},
\be\label{eq:HLA:xy}
W_y \log y=(1-\eps)c\,\xi\log^2 x+O_c(\log_2 x).
\ee
In particular, \eqref{eq:xy} holds, with $\alpha,\beta$ depending on $c$.
Also, from \eqref{eq:By-bounds} we have
\be\label{HLA:y-bounds}
(c/2) \log^2 x \le y=o((\log^2 x)\log_2 x).
 \ee
Let
\[
w_1 \defeq (y/\log y)^{1/2}, \qquad w_5\defeq y^8, \qquad z\defeq z(x/2).
\] 

Again, let $\cN\defeq (x/2,x]$, and define $C$ as in the previous proof.
We apply Lemma~\ref{lem:UK} with
\[
K\defeq 2\fl{\frac{100Cy}{\log x}}-1,
\]
so that $K\le \frac{200 C y}{\log x}$.
Similarly to \eqref{UK1} we get that
\be\label{UK2}
U_K \ge \int_{x/2}^{x} \P(S_{z(t)}=0) - 
\E_{z(t)}\binom{S_{z(t)}}{K+1}\, dt+O(E),
\ee
where, because the function $\fS_{z(t)}(\cH)$ appears 
already in \eqref{eq:HLA*}, as does the averaging over $\cH$,
we have
\be\label{HLA:E}
E \ll K x^{1-c} \ll x^{1-c}\log^2 x.
\ee
By the same reasoning as in the proof of Theorem~\ref{thm:HLgaps},
we get that
\be\label{HLA:Ez}
\E_{z(t)}\binom{S_{z(t)}}{K+1} \ll e^{-K} \ll x^{-10c},
\ee
where we used \eqref{HLA:y-bounds} in the last step.

Let $w\defeq y^8$
and fix $\cA_w$. By Lemma~\ref{lem:SzSP} we have
\be\label{HLA:PSz0}
\P_{w,z}(S_{z}=0)=(1-\Theta_{w,z})^{S_w}(1+\red{O(y^{-5})}).
\ee

Now put $w_1\defeq(y/\log y)^{1/2}$,
and let $\cA_{w_1}$ be fixed such that $S_{w_1}=W_y$.
This occurs with probability $\ge x^{-o(1)}$,
since $(y/\log y)^{1/2}=o(\log x)$ by \eqref{HLA:y-bounds}.
Conditional on $\cA_{w_1}$,
Corollary~\ref{cor:w1w5} implies that with probability
at least $1-O(x^{-100})$ we have
\[
S_w=(\tfrac{1}{16}+O(\eta))S_{w_1}=(\tfrac{1}{16}+O(\eta))W_y,
\]
where $\eta\defeq\frac{\log_4 x}{\log_3 x}$ as before
and the implied constants may depend on $c$.
Now fix $w$ such that the above holds.
Since
\[
\Theta_{w,z}=(1+O(\eta)) \frac{16\log y}{\xi\log x},
\]
\eqref{eq:HLA:xy} implies that
\[
\Theta_{w,z}S_w=(1+O(\eta))
(1-\eps) c \log x,
\]
where we have used \eqref{eq:HLA:xy} in the last step.
Inserting this last estimate into \eqref{HLA:PSz0}, we conclude that
\be\label{HLA:PSz}
\P_z(S_z=0) \gg e^{-(1+O(\eta))(1-\eps)c\log x} \gg x^{-(1-\eps/2)c}
\ee
In particular, the right side
of \eqref{HLA:PSz} has larger order than the right sides
in \eqref{HLA:E} and \eqref{HLA:Ez}. Thus, inserting
\eqref{HLA:E}, \eqref{HLA:Ez} and \eqref{HLA:PSz} into
\eqref{UK2}, we conclude that $U_K\ge 1$ if $x$ is sufficiently large
depending on $\eps$. By a simple diagonalization argument, the same
claim then holds for some $\eps=\eps(x)=o(1)$ going to zero
sufficiently slowly as $x \to \infty$. This completes the
proof of Theorem~\ref{thm:HLgaps-avg}.
\end{proof}


{\Large\section{The influence of exceptional zeros}
\label{sec:exceptional_zeros}}

In this section, we show that the existence of
exceptional zeros implies that $W_y$ 
is rather smaller than the upper bound in 
\eqref{eq:By-bounds} infinitely often.

\bigskip

\begin{theorem}\label{Wydeltaq}
Let $q\in\NN$, and suppose that there is a real Dirichlet character
$\chi_q$ mod $q$ such that $L(1-\delta_q,\chi_q)=0$ and
$0<\delta_q \le \frac{c}{\log q}$, where $c\defeq 1/11^2$. For
\be\label{ydeltaq}
y\defeq\exp \bigg\{ \pfrac{\log q}{\delta_q}^{1/2} \bigg\}
\ee
we have
\[
W_y \ll \delta_q y \change{=\frac{y\log q}{\log^2 y}. }
\]
\end{theorem}

\begin{proof}
\change{
Denote by $\pi(x;q,a)$ the number of primes $p\le x$ lying in the 
progression $a\bmod q$. By hypothesis, $qy \ge q^{1+1/\sqrt{c}}=q^{12}$,
therefore we may apply \cite[Corollary 1.4]{TZ}, obtaining
\begin{align*}
\pi(qy+1;q,1) &\le \sqrt{y/q}+\frac{2}{\log(qy)} \ssum{\sqrt{qy}<p\le qy \\ p\equiv 1\pmod{q}}
\log p \\
&\ll \sqrt{y/q}+\frac{\lambda qy}{\phi(q) \log(qy)}, 
\end{align*}
where
\[
\lambda \defeq 1 - (qy)^{-\delta_q}/(1-\delta_q) \ll \delta_q \log (qy).
\]
By Siegel's Theorem \cite[\S 21]{Dav}, for any $\eps>0$,
$\delta_q \gg_\eps q^{-\eps}$.
We conclude that
\[
\pi(qy+1;q,1) \ll \frac{\delta_q qy}{\phi(q)}.
\]
This may also be deduced from Gallagher's prime number theorem \cite[Theorem 7]{Gal70}.
}

Define the residue classes $a_p$ by $qa_p+1\equiv 0\bmod{p}$
when $p\nmid q$. \change{ Let $\cT$ denote the set of $n\le y$ with
 $n\not\equiv a_p\bmod{p}$ for all $p\nmid q$ such that $p\le \sqrt{y/\log y}$.
Then for any $n\in \cT$,} $qn+1$
is either prime or the product of two primes $>\sqrt{y/\log y}$.
Then we make a greedy choice of $a_p$ for $p\mid q$, \change{choosing successively
$a_p$ so that $a_p\bmod p$ covers a proportion at least $1/p$ of the remaining elements of $\cT$.}
 This shows that
\begin{align*}
W_y &\le \frac{\phi(q)}{q} |\cT| \\
&\le \frac{\phi(q)}{q} \bigg[ \pi(qy+1;q,1)+\sum_{\sqrt{y/\log y} < p \le \sqrt{qy+1}} \pi\( \frac{qy+1}{p};q,\overline{p}\) \bigg],
\end{align*}
\change{where $\overline{p}$ is the inverse of $p$ modulo $q$. 
Siegel's theorem implies} that $\log y \le q^{o(1)}$.
Applying the Brun-Titchmarsh theorem to the sum over $p$, we see that
\begin{align*}
W_y \ll \frac{\phi(q)}{q} \bigg[ \frac{qy\delta_q}{\phi(q)}+\frac{qy\log(q\log y)}{\phi(q)\log^2 y}\bigg]
\ll y \Big[ \delta_q+\frac{\log q}{\log^2 y} \Big]
\ll \delta_q y.
\end{align*}
This completes the proof.
\end{proof}

\begin{proof}[Proof of Theorem~\ref{thm:exceptional_zeros}]
Let $q\in Q$, and apply Thereom \ref{Wydeltaq} with
$y=y_q$ defined by \eqref{ydeltaq}.
By assumption, $\frac{\log y_q}{\log q} \to \infty$
as $q\to \infty$, and hence that
\[
\delta_q=\frac{\log q}{\log^2 y_q}=o\pfrac{1}{\log y_q}.
\]
This shows that $W_{y_q}=o (y_q/\log y_q)$, and the remaining
parts of Theorem~\ref{thm:exceptional_zeros} follow immediately.
\end{proof}


\end{document}